\documentclass[a4paper,10pt]{article}
\usepackage{courier}
\usepackage{blindtext} 
\usepackage{amsmath}
\usepackage{amssymb}
\usepackage[colorlinks=true,breaklinks]{hyperref}
\usepackage[hyphenbreaks]{breakurl}
\usepackage{xcolor}
\definecolor{c1}{rgb}{0,0,1}
\definecolor{c2}{rgb}{0,0.3,0.9}
\definecolor{c3}{rgb}{0.3,0.9}
\hypersetup{linkcolor={c1},citecolor={c2},urlcolor={c3}}
\usepackage{enumerate}
\usepackage{todonotes}
\usepackage{makeidx}
\usepackage{verbatim}
\makeindex
\usepackage[top=4cm,bottom=4cm,left=2.8cm,right=2.8cm]{geometry}
\usepackage{fancyhdr}

\newcommand{\bbR}{\mathbb{R}}
\newcommand{\R}{\bbR}

\newcommand\blfootnote[1]{%
  \begingroup
  \renewcommand\thefootnote{}\footnote{#1}%
  \addtocounter{footnote}{-1}%
  \endgroup
}

\def\XXint#1#2#3{{\setbox0=\hbox{$#1{#2#3}{\int}$ }
\vcenter{\hbox{$#2#3$ }}\kern-.6\wd0}}

\usepackage{amsthm}

\theoremstyle{plain}
\newtheorem{theorem}{Theorem}[section]

\theoremstyle{definition}

\theoremstyle{lemma}
\newtheorem{lemma}[theorem]{Lemma}

\theoremstyle{Remark}
\newtheorem{Remark}[theorem]{Remark}

\theoremstyle{proposition}
\newtheorem{proposition}[theorem]{Proposition}

\theoremstyle{corollary}
\newtheorem{corollary}[theorem]{Corollary}

\theoremstyle{example}
\newtheorem{example}[theorem]{Example}

\theoremstyle{assumption}
\newtheorem{assumption}[theorem]{Assumption}

\newcommand{\aaa}{\textcolor{red}}

\usepackage{systeme}
\usepackage{colonequals}

\begin{document}
\pagestyle{empty}
\title{Exponential stability and hypoelliptic regularization for the kinetic Fokker-Planck equation with confining potential}
\author{Anton Arnold, 
 Gayrat Toshpulatov
 }

\maketitle
\blfootnote{\textit{Authors addresses}: Institute for Analysis and Scientific Computing, TU Wien,
Wiedner Hauptstraße 8–10, 1040 Vienna, Austria, 
 {\tt anton.arnold@tuwien.ac.at}, {\tt gayrat.toshpulatov@tuwien.ac.at} }
 \begin{small}
 \begin{center}
 
 \end{center}
 \end{small}
\pagestyle{plain}

\begin{abstract}
This paper is concerned with a modified entropy method to establish the large-time convergence towards the (unique) steady state, for  kinetic Fokker-Planck equations with non-quadratic confinement potentials in whole space. We extend previous approaches by analyzing Lyapunov functionals with non-constant weight matrices in the dissipation functional (a generalized Fisher information). We establish exponential convergence in a weighted $H^1$-norm with rates that become sharp in the case of quadratic potentials. In the defective case for quadratic potentials, i.e.\ when the drift matrix has non-trivial Jordan blocks, the weighted $L^2$-distance between a Fokker-Planck-solution and the steady state has always a sharp decay estimate of the order $\mathcal O\big( (1+t)e^{-t\nu/2}\big)$, with $\nu$ the friction parameter. The presented method also gives new hypoelliptic regularization results for kinetic Fokker-Planck equations (from a weighted $L^2$-space to a weighted $H^1$-space).
\end{abstract}
\begin{small}
 \textbf{Keywords:} Kinetic theory, Fokker-Planck equation, confinement potential, degenerate evolution, long time behavior,  convergence to equilibrium, hypocoercivity,  hypoellictic regularity, Lypunov functional.\\
\textbf{2020 Mathematics Subject Classification:}  35Q84, 
 35B40, 
 35Q82, 
 82C40. 
\end{small}
\tableofcontents
\newpage

\section{Introduction}

This paper is devoted to the study of the long time behavior of the kinetic Fokker-Planck equation
\begin{equation}\label{KFP}\begin{cases}
\partial_t f+v \cdot \nabla_x f-\nabla_x V \cdot \nabla_v f=\nu  \text{div}_v(vf)+\sigma \Delta_v f,\,   \,  \,  x,\,v \in \mathbb{R}^n,\, t>0\\
f(t=0)=f_0 \in L^1(\mathbb{R}^{2n})
\end{cases}
\end{equation}
describing the time evolution of the phase space probability density $f(t,x,v),$ e.g. in a plasma \cite{H.Risken}. Applications range from plasma physics \cite{Lib, F.P} to stellar dynamics \cite{chan1, chan}. Here $V=V(x)$ is a given smooth, bounded below confinement potential for the system, and $\nu>0, \,\sigma>0$ denote the friction and diffusion parameters, respectively.
This equation is associated with the Langevin stochastic differential equation
\begin{equation*}
\begin{cases}
dx_t=v_tdt\\
dv_t=-\nu v_t dt-\nabla V(x_t)dt+\sqrt{2\sigma} dB_t,
\end{cases}
\end{equation*}
 where $\{B_t\}_{t\geq 0}$ is a Brownian motion in $\mathbb{R}^n$ with covariance $\langle B_t,B_{t'}\rangle=\delta_{t-t'}.$\\ 
 Since the equation conserves mass, i.e.,
$$\int_{\mathbb{R}^{2n}}f(t,x,v) dxdv=\int_{\mathbb{R}^{2n}}f_0(x,v) dxdv,\, \, \, \, \, t\geq 0,$$
  we shall always assume (without restriction of generality) that $\displaystyle \int_{\mathbb{R}^{2n}}f_0(x,v) dxdv=1.$ The unique normalized steady state of \eqref{KFP} is given by 
\begin{equation}\label{steady.s} f_{\infty}(x,v)=c_Ve^{-\frac{\nu}{\sigma}[V(x)+\frac{|v|^2}{2}]},\, \,  \, x,v \in \mathbb{R}^n,
\end{equation}
where $c_V$ is a positive constant such that $\int_{\mathbb{R}^{2n}} f_{\infty}(x,v)dxdv=1.$
The following equation is also considered as the kinetic Fokker-Planck equation:
\begin{equation}\label{KFP2}\partial_t h+v \cdot \nabla_x h-\nabla_x V \cdot \nabla_v h=\sigma \Delta_v h-\nu v\cdot\nabla_v h,\,   \,  \,  x,\,v \in \mathbb{R}^n,\, t>0,
\end{equation}
and to switch from \eqref{KFP} to \eqref{KFP2} it suffices to set $h:=f/f_{\infty}.$

It was shown in \cite{Hel.Nier} that, if $V\in C^{\infty}(\mathbb{R}^n),$ \eqref{KFP2}
 generates a $C^{\infty}$ regularizing contraction semigroup in $L^2(\mathbb{R}^d, f_{\infty}):=\{g\colon \mathbb{R}^d \to \mathbb{R}: g \text{ is measurable and  } \int_{\mathbb{R}^d}g^2f_{\infty}dxdv< \infty \}, \, \, \, d=2n.$ For well-posedness with non-smooth potentials, we refer to \cite[Theorem 6, Theorem 7]{Vil}.

 The long time behavior and exponential convergence of the solution to the steady state has been studied and there are various results: in \cite{Des.Vil}, algebraic decay was proved for potentials that are asymptotically
quadratic (as $|x|\to \infty$) and for initial conditions that are bounded below and above by Gaussians. The authors used logarithmic Sobolev inequalities and entropy methods. In \cite{Herau.Niear}, exponential decay was obtained also for faster growing potentials and more general initial conditions. That proof is based on hypoellipticity techniques.  By using hypoelliptic methods, Villani proved exponential convergence results in $H^1(\mathbb{R}^d, f_{\infty}):=\{g \in L^2(\mathbb{R}^d, f_{\infty}): |\nabla g| \in L^2(\mathbb{R}^d, f_{\infty})\}$ \cite[Theorem 35]{Vil} and in $L^2(\mathbb{R}^d, f_{\infty})$ \cite[Theorem 37]{Vil}. The main conditions in  Villani's theorems above, as well as in \cite{DMSch, Baudoin, BGH, H^k, BS, CLW}, 
  are the validity of the Poincar\'e inequality \eqref{Poincare}
 and  the criterion 
\begin{equation}\label{Villanis condition}
\exists \, \, C\geq 0\, \, : \, \,  \,  \, \, \, \, \, 
 \begin{small}\left|\left|\frac{\partial^2  V(x)}{\partial x^2}\right|\right|\end{small}\leq C(1+|\nabla V(x)|),\, \, \, \, \forall x \in \mathbb{R}^n,
\end{equation}
where \begin{small}$\displaystyle\left|\left|\frac{\partial^2  V(x)}{\partial x^2}\right|\right|$\end{small} denotes the Frobenius norm of \begin{small}$\displaystyle\frac{\partial^2  V(x)}{\partial x^2}.$\end{small}\\
When $\frac{\partial^2  V}{\partial x^2}$ is bounded, Villani also proved that the solution converges to  the steady state exponentially in the logarithmic entropy \cite[Theorem 39]{Vil}. This result was extended  in \cite{CAMZ} to potentials $V$ satisfying a weighted log-Sobolev inequality and the condition that $V^{-2\eta}\frac{\partial^2  V}{\partial x^2}$ is bounded for some $\eta\geq 0.$ Even though Villani's result allows for a general class of potentials, the growth condition \eqref{Villanis condition} is not satisfied by potentials with singularities. This type of potentials, such as Lennard-Jones type interactions with confinement, are considered in \cite{BGH} and their method relies on an explicit construction of a Lypunov function and Gamma calculus. In \cite{DMSch}, Dolbeault, Mouhot, and Schmeiser  developed a method to get exponential decay in $L^2$  for a large class of linear kinetic equations, including \eqref{KFP}. Their method was also used to study the long time behavior of \eqref{KFP} when the potential $V$ is zero or grows slowly as $|x|\to \infty,$ see \cite{HnoC, HweakC}. Based on a probabilistic coupling method, Eberle, Guillin, and Zimmer \cite{Eber} obtained an exponential decay result in Wasserstein distance.

 The associated semigroup of the kinetic Fokker-Planck equation has instantaneous regularizing properties  which is called \textit{hypoellipticity} \cite{hypoell}. This hypoelliptic regularization  is obvious when the confining potential $V$ is zero or quadratic as  the fundamental solution can be explicitly computed (see \cite{Kol}, \cite{hypoell}). For potentials such that $\frac{\partial^2  V}{\partial x^2}$ is bounded, H\'erau \cite{Herau} obtained  short time estimates for a $L^2(\mathbb{R}^d, f_{\infty})\to H^1(\mathbb{R}^d, f_{\infty})$ regularization by constructing a suitable Lyapunov functional.  Based on interpolation inequalities and a system of differential inequalities, Villani \cite[Appendix A.21]{Vil} extended H\'erau's result for potentials satisfying \eqref{Villanis condition}.

  We provide  a new method to establish exponential decay of the solution to the steady state  in $H^1(\mathbb{R}^d, f_{\infty}) $ for a wide class of potentials: Our method extends \cite{Vil, FAS, AE} by allowing for more general Lyapunov functionals. Generalizing the previous approaches, the weight matrix in the dissipation functional (a generalized Fisher information) may now depend on $x$ and $v.$ This
   leads to a new criterion  on the potential $V.$ 
   {For this entropy method we need the time derivative of the dissipation functional, but we also provide its $(x,v)$--pointwise analog, in the spirit of the \emph{Gamma calculus} \cite{Baudoin}.}
   We provide a formula to estimate easily  the exponential decay rate depending on the parameters of the equation, the constants appearing in the Poincar\'e inequality \eqref{Poincare} and the growth condition on the potential (see \eqref{Condition1} below). As a test of the effectiveness of our method, we show that our estimate on the decay rate is sharp when the potential is a quadratic polynomial. Moreover, our method lets us  obtain estimates on the hypoelliptic regularization for potentials that are more general than in \cite{Herau}.
  
  The organization of this paper is as follows. In Section 2, we define the assumptions on the potential, state the main results, and present concrete examples of such potentials. In Section 3, we present the intuition and explain our method. Section 4 contains  important lemmas about matrix inequalities which are important to construct  suitable Lyapunov functionals. The final section presents the proof of the main results.

\section{Main results }  
    We make the following assumptions.
    \begin{assumption}\label{Assum:Poincare}  There exists a constant $C_{PI}>0$ such that the Poincar\'e inequality 
    \begin{equation}\label{Poincare}
  \int_{\mathbb{R}^{2n}} h^2 f_{\infty}dxdv-\left(\int_{\mathbb{R}^{2n}} h f_{\infty}dxdv\right)^2\leq \frac{1}{C_{PI}}\int_{\mathbb{R}^{2n}} (|\nabla_x h|^2+|\nabla_v h|^2)f_{\infty}dxdv\, \, 
\end{equation}
holds for all $h \in H^1(\mathbb{R}^d, f_{\infty}).$
\end{assumption}
Sufficient conditions on the potential appearing in  $f_{\infty}$ so that the Poincar\'e inequality holds, e.g. the Bakry-Emery criterion, are presented in \cite[Chapter 4]{AGMar}. 
  \begin{assumption}\label{A1} There are constants $c\in \mathbb{R}$ and $\tau\in [0,\nu) $ such that the following $\mathbb{R}^{m\times m}$  matrix, $ m\colonequals n(n+1),$
 \begin{equation}\label{Condition1}  \begin{pmatrix}
 \nu\left(\frac{\partial^2  V(x)}{\partial x^2}+cI\right)&0&...&0&-\frac{1}{2}\frac{\partial^2(\partial_{x_1}V(x))}{\partial x^2} \\
 0&\nu \left(\frac{\partial^2  V(x)}{\partial x^2}+cI\right)&...&0&-\frac{1}{2}\frac{\partial^2(\partial_{x_2}V(x))}{\partial x^2}\\ 
 ...&...&...&...&...\\
 0&0&...&\nu \left(\frac{\partial^2  V(x)}{\partial x^2}+cI\right)&-\frac{1}{2}\frac{\partial^2(\partial_{x_n}V(x))}{\partial x^2}\\
 -\frac{1}{2}\frac{\partial^2(\partial_{x_1}V(x))}{\partial x^2}&-\frac{1}{2}\frac{\partial^2(\partial_{x_2}V(x))}{\partial x^2}&...&-\frac{1}{2}\frac{\partial^2(\partial_{x_n}V(x))}{\partial x^2}&\frac{\tau \nu}{2\sigma}\left(\frac{\partial^2  V(x)}{\partial x^2}+cI\right)
 \end{pmatrix}
 \end{equation}
 is positive semi-definite for all $x \in \mathbb{R}^n,$ where $I\in \mathbb{R}^{n \times n}$ denotes the identity matrix. 
 \end{assumption}
  Roughly speaking, Assumption \ref{A1} essentially means that the second order derivatives of $V$ control the third order ones. It implies that $\frac{\partial^2 V(x)}{\partial x^2}+ cI$ is  positive semi-definite for all $x \in \mathbb{R}^n,$ and hence the eigenvalues of $\frac{\partial^2 V(x)}{\partial x^2}$ are uniformly bounded from below. We note that, in contrast to the Bakry-Emery strategy \cite{BE1}, the eigenvalues here may take negative values.
  
 Let  $\displaystyle \alpha(x)\in \mathbb{R}$  denote the smallest eigenvalue of $\frac{\partial^2 V(x)}{\partial x^2}$ at $x \in \mathbb{R}^n.$ Then the following condition implies Assumption \ref{A1}. For its proof see  Appendix \ref{6.1}.\\
\textbf{Assumption 2.{2'.}}\label{A2}
 \textit{There are constants $c\in \mathbb{R}$ and $\tau\in [ 0,\nu) $ such that } $\frac{\partial^2  V(x)}{\partial x^2}+cI$ \textit{is positive semi-definite and}\footnote{For two matrices $A$ 
 and $B\in \mathbb{R}^{n \times n} ,$  $A\geq B$ means that $A-B$ is positive semi-definite.}
  \begin{equation}\label{Condition2}
  -\sqrt{\frac{2\tau \nu^2 }{n\sigma}}(\alpha(x)+ c)I\leq  \frac{\partial^2 (\partial_{x_i}V(x))}{\partial x^2} \leq \sqrt{\frac{2\tau \nu^2}{n \sigma}}(\alpha(x)+ c)I
  \end{equation}
  \textit{ for all  }  $ x \in \mathbb{R}^n$ \textit{ and} $ i\in \{1,...,n\}.$  \\

  We denote
  \begin{equation}\label{alpha-zero}\displaystyle \alpha_0\colonequals \inf_{x \in \mathbb{R}^n }\alpha(x)
  \end{equation} {and assume in the sequel that $\alpha_0>- \infty.$ Hence Assumption \ref{A1} can only hold for some $c\geq-\alpha_0.$}  
  
In the following results, we require that $\frac{f_0}{f_{\infty}}\in L^2(\mathbb{R}^{2n}, f_{\infty})$ which implies  $f_0\in  L^{1}(\mathbb{R}^{2d})$ because of the H\"older inequality  $ \int_{\mathbb{R}^{2n}}f_0 dxdv\leq\sqrt{(\int_{\mathbb{R}^{2n}}\frac{f^2_0}{f_{\infty}}dxdv)(\int_{\mathbb{R}^{2n}}f_{\infty}dxdv)} $   and $\int_{\mathbb{R}^{2n}} f_{\infty}dxdv=1.$
We now state our first result, i.e. exponential decay of a functional that is a linear combination of the weighted $L^2-$norm and a Fisher information-type functional:
  \begin{theorem}\label{Main}
  Let $V $ be a $C^{\infty}$ potential in $\mathbb{R}^n$ satisfying Assumptions \ref{Assum:Poincare} and \ref{A1}. Let $C_{PI},$ $c,$ $\tau,$ and $\alpha_0$ be the constants in \eqref{Poincare}, \eqref{Condition1}, and \eqref{alpha-zero}. Suppose  the initial data $f_0$  satisfies   {\small$\displaystyle \frac{f_0}{f_{\infty}}\in H^1(\mathbb{R}^{2n}, f_{\infty})$} and   \begin{small}$ \displaystyle \int_{\mathbb{R}^{2n}} \nabla^T_v \left(\frac{f_0}{f_{\infty}}\right)\frac{\partial^2V}{\partial x^2}\nabla_v \left(\frac{f_0}{f_{\infty}}\right) f_{\infty}dxdv<\infty.$\end{small}  Then there are explicitly computable constants $C>0$ and $\lambda>0$ (independent of $f_0$) such that the solution $f(t)$ of \eqref{KFP} satisfies \begin{multline}\label{main.est} 
\int_{\mathbb{R}^{2n}}\left(\frac{f(t)}{f_{\infty}}-1\right)^2 f_{\infty}dxdv +\int_{\mathbb{R}^{2n}}\left|\nabla_{x}\left(\frac{f(t)}{f_{\infty}}\right)\right|^2 f_{\infty}dxdv\\+\int_{\mathbb{R}^{2n}}\nabla^T_{v}\left(\frac{f(t)}{f_{\infty}}\right)\left(\frac{\partial^2 V}{\partial x^2}+(1-\alpha_0)I\right)\nabla_{v}\left(\frac{f(t)}{f_{\infty}}\right)f_{\infty} dxdv \\
   \leq C e^{-2\lambda t} \left[\int_{\mathbb{R}^{2n}}\left(\frac{f_0}{f_{\infty}}-1\right)^2 f_{\infty}dxdv+\int_{\mathbb{R}^{2n}}\left|\nabla_{x}\left(\frac{f_0}{f_{\infty}}\right)\right|^2 f_{\infty}dxdv\right.\\+ \left.\int_{\mathbb{R}^{2n}}\nabla^T_{v}\left(\frac{f_0}{f_{\infty}}\right)\left(\frac{\partial^2 V}{\partial x^2}+ (1-\alpha_0) I\right)\nabla_{v}\left(\frac{f_0}{f_{\infty}}\right)f_{\infty}dxdv\right]
  \end{multline}
  for all $t\geq 0.$ Moreover, we have: 
 \begin{enumerate}
 \item[(a)]
  if {\small {$\displaystyle  \alpha_0>\frac{\nu^2}{4},\, \, \, c\leq-\frac{\nu^2}{4}, $}} then {\small $\displaystyle 2\lambda={\nu-\tau};$}
\item[(b)]
if {\small $\displaystyle c=-\alpha_0=-\frac{\nu^2}{4},$}  then {\small $\displaystyle 2\lambda={\nu-\tau-\varepsilon}$} for any {\small $\displaystyle \varepsilon\in (0,\nu-\tau);$}  
  \item[(c)]
   if {\small $\displaystyle c >-\frac{\nu^2}{4}, \,\, c+2\alpha_0> \frac{\nu^2}{4},$} then  {\small $$\displaystyle 2\lambda=\begin{cases} \displaystyle   \nu-\tau-\frac{c+\frac{\nu^2}{4}}{\sqrt{c+\alpha_0}}  & \text{ if } \nu-\tau\geq A^{-1}_1+\frac{c+\frac{\nu^2}{4}}{\sqrt{c+\alpha_0}}  \\  \displaystyle 
\frac{(\nu-\tau)\sqrt{c+\alpha_0}-(c+\frac{\nu^2}{4}) (\sqrt{1+s_1^2}-s_1)}{{\sqrt{c+\alpha_0}+A_1s_1(c+\frac{\nu^2}{4}})  } & \text{ if }  \nu-\tau< A^{-1}_1+\frac{c+\frac{\nu^2}{4}}{\sqrt{c+\alpha_0}}   \end{cases},$$}
where {\small $ \displaystyle  A_1:=\frac{1+ \frac{\nu^2}{4}+c+\alpha_0+ \sqrt{( \frac{\nu^2}{4}+c+\alpha_0-1)^2+\nu^2}}{2\sigma C_{PI}},$\\ $s_1:=\begin{cases} 
 \frac{A^2_1(c+\frac{\nu^2}{4})^2-c-\alpha_0}{2A_1(c+\frac{\nu^2}{4}) \sqrt{c+\alpha_0}}& \text{ if }  (\nu-\tau)A_1=2  \\  
\frac{1}{\nu-\tau}\left[\left|\frac{(\nu-\tau)A_1-1}{(\nu-\tau)A_1-2}\right|\sqrt{\frac{(c+\frac{\nu^2}{4})^2}{c+\alpha_0}+2(\nu-\tau)A^{-1}_1-(\nu-\tau)^2}-\frac{c+\frac{\nu^2}{4}}{((\nu-\tau)A_1-2)\sqrt{c+\alpha_0}}\right] & \text{ if }  (\nu-\tau)A_1\neq 2  
  \end{cases};$}
 \item[(d)]
  if {\small $\displaystyle c >-\frac{\nu^2}{4}, \,\, c+2\alpha_0\leq \frac{\nu^2}{4},$} then  
{\small $$\displaystyle 2\lambda=\begin{cases}\displaystyle  \nu-\tau-\sqrt{{\nu^2}-4\alpha_0} & \text{ if } \nu-\tau\geq A^{-1}_2+\sqrt{{\nu^2}-4\alpha_0} \\ \displaystyle 
\frac{\nu-\tau-\sqrt{{\nu^2}-4\alpha_0} (\sqrt{1+s_2^2}-s_2)}{1+A_2s_2\sqrt{{\nu^2}-4\alpha_0} } & \text{ if }  \nu-\tau< A^{-1}_2+\sqrt{{\nu^2}-4\alpha_0}  \end{cases},$$}
where {\small $\displaystyle A_2:=\frac{1+ \frac{\nu^2}{2}-\alpha_0+ \sqrt{( \frac{\nu^2}{2}-\alpha_0-1)^2+\nu^2}}{2\sigma C_{PI}}, $\\ $ \displaystyle s_2:=\begin{cases}
 \frac{A^2_2 ({\nu^2-4\alpha_0})-1}{2A_2\sqrt{{\nu^2}-4\alpha_0}}& \text{ if }  (\nu-\tau)A_2=2  \\  
\frac{1}{\nu-\tau}\left[\left|\frac{(\nu-\tau)A_2-1}{(\nu-\tau)A_2-2}\right|\sqrt{{{\nu^2}-4\alpha_0}+2(\nu-\tau)A^{-1}_2-(\nu-\tau)^2}-\frac{\sqrt{{\nu^2}-4\alpha_0}}{((\nu-\tau)A_2-2)}\right] & \text{ if }  (\nu-\tau)A_2\neq 2  
  \end{cases};$}
\item[(e)]
 if $V(x)$ is a quadratic polynomial of $x$ and $\frac{\partial^2 V}{\partial x^2}$ is positive definite, then  Assumptions \ref{Assum:Poincare} and \ref{A1} are satisfied  with $\tau=0,$ $c=-\alpha_0$  (this rules out the conditions in the case of (c)). Moreover, the decay rates $\lambda$ in (a) and  (d) are  sharp and, in the case of (d),  $\nu\geq {A^{-1}_2}+\sqrt{{\nu^2}-4\alpha_0} $ holds and so $2\lambda= \nu-\sqrt{{\nu^2}-4\alpha_0}.$ In the case of (b), the decay rate $2\lambda=\nu-\varepsilon$ is sharp in the sense that  \eqref{main.est}  holds with the rate $2\lambda=\nu-\varepsilon$  for any small fixed $\varepsilon\in (0,\nu),$ but it does not hold with the rate  $2\lambda=\nu.$ 

 \end{enumerate}
\end{theorem}
\begin{Remark} 
\begin{enumerate}
\item  It is possible to make weaker regularity hypothesis on the potential $V,$ but we maintain the assumption that $V \in C^{\infty} $ to keep the presentation simple.

\item  Since  $\frac{\partial^2 V}{\partial x^2}+ (1-\alpha_0) I \geq I,$ \eqref{main.est} implies that the solution converges exponentially to the steady state  in $H^1(\mathbb{R}^{2n}, f_{\infty}).$    If the eigenvalues of  $\frac{\partial^2 V}{\partial x^2}$ are uniformly bounded, then \eqref{main.est} is equivalent to the exponential decay of the solution to the steady state in $H^1(\mathbb{R}^{2n}, f_{\infty}).$ Due to the Poincar\'e inequality \eqref{Poincare}, the $L^2-$term on the right hand side of \eqref{main.est} could be omitted. 
\item  If $V $ satisfies Assumption \ref{A1} with some constants  $c\in \mathbb{R} $ and $\tau \in [0,\nu),$ then $V $ also satisfies Assumption \ref{A1} with any $\tilde{c}\geq c$ and $\tilde{\tau}\in [\tau,\nu).$ Therefore, these constants are not unique. But the exponential decay rate $\lambda$ obtained in Theorem \ref{Main} depends on the choice of $c$ and $ \tau.$ To obtain a better rate, one has to optimize $\lambda=\lambda(c,\tau)$   with respect to all $c$ and $\tau$ satisfying Assumption \ref{A1}.
\item In Theorem \ref{Main} $(b),$  the constant $C$ in \eqref{main.est} depends on $\varepsilon,$ and $C=C(\varepsilon)\to \infty$ as $\varepsilon \to 0.$ 

\item  The highest exponential rate is $\frac{\nu}{2}$ which can be attained by the quadratic potentials $V$ with $\frac{\partial^2 V}{\partial x^2}\geq \frac{\nu^2}{4}I.$ 
\end{enumerate}
\end{Remark}

 When $V$ is a quadratic polynomial as in Theorem \ref{Main} $(e),$ we prove  the following sharp estimates.
\begin{proposition}\label{prop.}
Let $V$ be a quadratic polynomial and $\frac{\partial^2 V}{\partial x^2}$ be positive definite.  Let $\alpha_0>0$ be the smallest eigenvalue  of  $\frac{\partial^2 V}{\partial x^2},$ then\footnote{For functions $\varphi=\varphi(t)$ and $ \phi=\phi(t),$   $\varphi \asymp \phi$ as $t\to \infty$ means  $\varphi=\mathcal{O}(\phi)$ and $\phi=\mathcal{O}(\varphi)$ as $t\to \infty.$}   
\begin{equation}\label{l^2 t^2}
\sup_{1\neq \frac{f_0}{f_{\infty}} \in L^2(\mathbb{R}^d, f_{\infty})}\frac{||f(t)/f_{\infty}-1||_{L^2(\mathbb{R}^d, f_{\infty})}}{||f_0/f_{\infty}-1||_{L^2(\mathbb{R}^d, f_{\infty})}}\asymp \begin{cases}
e^{-\frac{\nu}{2}t},& \text{ if } \alpha_0>\frac{\nu^2}{4} \\
(1+ t)e^{-\frac{\nu}{2}t},& \text{ if } \alpha_0=\frac{\nu^2}{4}\\
 e^{-{\frac{\nu-\sqrt{\nu^2-4\alpha_0}}{2}t}},& \text{ if } \alpha_0<\frac{\nu^2}{4} 
\end{cases} \, \, \, \, \, \text{ as } \, \, t\to \infty.
   \end{equation}
 
\end{proposition}
  We shall use this proposition to prove the  sharpness of the decay rates in Theorem \ref{Main} $(e).$ When $V$ is a quadratic polynomial and $ -\alpha_0=-\frac{\nu^2}{4}\equalscolon c, $ Theorem \ref{Main} $(e)$ shows that the decay  in \eqref{main.est} can be $e^{-(\nu-\varepsilon)t}$ for any small  fixed $\varepsilon\in (0,\nu),$ but it can not be $e^{-\nu t}.$ In this case, it is natural to expect a decay between $e^{-\nu t}$  and $e^{-(\nu-\varepsilon)t}:$ Proposition \ref{prop.} shows that this is indeed the case for the square of the $L^2-$norm, with the decay  $(1+ t)^2e^{-{\nu}t}.$  But an analogous extension of  this result  for the functional on the left hand side of \eqref{main.est} (i.e., to replace the  term $C e^{-(\nu-\varepsilon) t}$  with {$C(1+t)^2e^{-\nu t}$ }) has not been obtained so far.   
\begin{Remark}
  Under assumptions of Proposition \ref{prop.}, 
we can construct special solutions  $f_s(t)$ (see \cite[Section 6]{AE}) which satisfy  $$\frac{||f_s(t)/f_{\infty}-1||_{L^2(\mathbb{R}^d, f_{\infty})}}{||f_0/f_{\infty}-1||_{L^2(\mathbb{R}^d, f_{\infty})}}\asymp \begin{cases}
e^{-\frac{\nu}{2}t},& \text{ if } \alpha_0>\frac{\nu^2}{4} \\
(1+ t)e^{-\frac{\nu}{2}t},& \text{ if } \alpha_0=\frac{\nu^2}{4}\\
 e^{-{\frac{\nu-\sqrt{\nu^2-4\alpha_0}}{2}t}},& \text{ if } \alpha_0<\frac{\nu^2}{4} 
\end{cases} \, \, \, \, \, \text{ as } \, \, t\to \infty.$$ 

\end{Remark}
Our next result is
about the estimates on the hypoelliptic regularization.
\begin{theorem}\label{Hyp.ellip.}
Assume $V $ is a $C^{\infty}$ potential on $\mathbb{R}^n$ and there are constants $c \in \mathbb{R}$ and $\tau\geq 0$ such that the matrix \eqref{Condition1} is positive semi-definite for all $x \in \mathbb{R}^n.$ Suppose  the initial data $f_0$  satisfies    \begin{small}$\displaystyle \int_{\mathbb{R}^{2n}} \left(\frac{f_0}{f_{\infty}}-1 \right)^2 \left(\left|\left| \frac{\partial^2V}{\partial x^2}\right|\right|^2+1\right) f_{\infty}dxdv<\infty.$\end{small}  Then, for any $t_0>0,$ there are explicitly computable constants $C_1=C_1(t_0)>0$ and $C_2=C_2(t_0)>0$ (independent of $f_0$) such that the inequalities
\begin{equation}\label{Vhyp.ell.1}
 \displaystyle \int_{\mathbb{R}^{2n}}\left|\nabla_x \left(\frac{f(t)}{f_{\infty}}\right) \right|^2f_{\infty} dx dv \leq \frac{C_1}{t^{3}} \int_{\mathbb{R}^{2n}} \left(\frac{f_0}{f_{\infty}} -1\right)^2 \left(\left|\left|\frac{\partial^2V}{\partial x^2}\right| \right|^2+1\right) f_{\infty}dxdv
\end{equation}
and 
\begin{multline}\label{Vhyp.ell.2}
 \displaystyle \int_{\mathbb{R}^{2n}}\nabla_v^T \left(\frac{f(t)}{f_{\infty}}\right)\left(\frac{\partial^2V}{\partial x^2}+(1-\alpha_0)I\right)\nabla_v \left(\frac{f(t)}{f_{\infty}}\right)  f_{\infty} dx dv\\ \leq \frac{C_2}{t}\int_{\mathbb{R}^{2n}} \left(\frac{f_0}{f_{\infty}} -1\right)^2 \left(\left| \left|\frac{\partial^2V}{\partial x^2}\right|\right|^2+1\right) f_{\infty}dxdv
\end{multline}
hold for all $t \in (0,t_0].$ 
\end{theorem}
 
In Theorem \ref{Main} we assumed that the initial data $f_0/f_{\infty}$ is in $H^1(\mathbb{R}^d, f_{\infty}).$ If we use the estimates in Theorem \ref{Hyp.ellip.},  this condition can be relaxed:
\begin{corollary}\label{Cor.}
Let $V $ be a $C^{\infty}$ potential in $\mathbb{R}^n$ satisfying Assumptions \ref{Assum:Poincare} and \ref{A1}. Suppose  the initial data $f_0$  satisfies    \begin{small} $ \displaystyle \int_{\mathbb{R}^{2n}} \left(\frac{f_0}{f_{\infty}} -1\right)^2 \left(\left|\left|\frac{\partial^2V}{\partial x^2}\right|\right|^2+1\right) f_{\infty}dxdv<\infty.$\end{small} Then, for any $t_0>0,$ there is an explicitly computable constant $C=C(t_0)>0$ (independent of $f_0$)  such that 
 \begin{multline}\label{main.est.hyp} 
\int_{\mathbb{R}^{2n}}\left(\frac{f(t)}{f_{\infty}}-1\right)^2 f_{\infty}dxdv +\int_{\mathbb{R}^{2n}}\left|\nabla_{x}\left(\frac{f(t)}{f_{\infty}}\right)\right|^2 f_{\infty}dxdv\\+\int_{\mathbb{R}^{2n}}\nabla^T_{v}\left(\frac{f(t)}{f_{\infty}}\right)\left(\frac{\partial^2 V}{\partial x^2}+(1-\alpha_0)I\right)\nabla_{v}\left(\frac{f(t)}{f_{\infty}}\right)f_{\infty} dxdv \\
   \leq C e^{-2\lambda t} \int_{\mathbb{R}^{2n}} \left(\frac{f_0}{f_{\infty}} -1\right)^2 \left(\left|\left|\frac{\partial^2V}{\partial x^2}\right|\right|^2+1\right) f_{\infty}dxdv
  \end{multline}
  holds {for all $t\geq t_0$} with $\lambda$ defined in Theorem \ref{Main}.
\end{corollary}

\begin{Remark} 
\begin{enumerate}

\item {In contrast  to Theorem \ref{Main},} Theorem \ref{Hyp.ellip.} holds even if the Poincar\'e inequality  \eqref{Poincare} is not satisfied by $f_{\infty}.$ Also, $\tau$ can be larger than $\nu.$ 
\item  The exponents of $t$ in \eqref{Vhyp.ell.1} and \eqref{Vhyp.ell.2}  are optimal when $V$ is a quadratic polynomial (see \cite[Appendix A]{H^k}).
\end{enumerate}  
\end{Remark}  

To illustrate our result, we present  concrete examples of potentials $V$ satisfying our Assumption \ref{Assum:Poincare} and Assumption \ref{A1}: 
\begin{example}[Polynomial confining potentials]\label{ex1}
\emph{
{$a)$} As mentioned in Theorem \ref{Main}, if   $V(x)=\frac{{x}^T M^{-1} x}{2}+p\cdot x+q, \,$  $x\in \mathbb{R}^n$ with a positive definite covariance matrix   $M^{-1}\in \mathbb{R}^{n\times n},$  a constant vector  $p\in \mathbb{R}^n$ and a constant $q\in \mathbb{R},$  the convergence rate is \begin{equation*}
\lambda=\begin{cases} 
\frac{\nu}{2},& \text{ if } \alpha_0>\frac{\nu^2}{4} \, \, \, \, \,\, \, \, \, \,  \,\, \, \, \, \, \,\, \, \, \, \, \,\, \, \, \, \, \,\, \, \, \, \, \,\, \, \, \, \, \,\, \, \, \, \, \,\, \, \, \, \, \,  \text{ (case (a))}\\
\frac{ \nu-\varepsilon}{2},& \text{ if } \alpha_0=\frac{\nu^2}{4}, \text{  for any } \varepsilon\in (0,\nu) \, \, \, \, \text{ (case (b))}\\
 \frac{\nu-\sqrt{\nu^2-4\alpha_0}}{2},& \text{ if } \alpha_0<\frac{\nu^2}{4} \,\, \, \, \, \,\, \, \, \, \,  \,\, \, \, \, \, \,\, \, \, \, \, \,\, \, \, \, \, \,\, \, \, \, \, \,\, \, \, \, \, \,\, \, \, \, \, \,\, \, \, \, \, \,  \text{ (case (d))} \end{cases},
 \end{equation*}  and it is sharp for $\alpha_0\neq \frac{\nu^2}{4},$ {where  $\alpha_0$  is} the smallest eigenvalue of $M^{-1}$ (see Theorem 2.3 $(e)$).}\\
   
\emph{$b)$ More generally, we consider potentials of the form $$V(x)=r|x|^{2k}+V_0(x)$$
where $r>0,$ $k \in \mathbb{N}$ and $V_0\colon \mathbb{R}^n \to \mathbb{R}$ is a polynomial of degree  $j<2k.$ Since we have already considered quadratic potentials, we assume $k\geq 2.$ { $V$ satisfies the Poincar\'e inequality \eqref{Poincare}; this can be proven, for example, by showing that $V$ satisfies one of the sufficient conditions given in \cite[Corollary 1.6]{Poin}.} Concerning Assumption \textcolor{blue}{2.}\ref{A2}\textcolor{blue}{'} we have
\begin{equation*} r\frac{\partial^2 |x|^{2k}}{\partial x^2}=2kr|x|^{2k-2}I+2k(2k-2)r|x|^{2k-4}\begin{pmatrix}
x_1^2& x_1x_2&...&x_1x_n\\
x_1x_2&x^2_2&...&x_2x_n\\
...&...&...&...\\
x_1x_n&x_2 x_n&...&x_n^2\\
\end{pmatrix}\geq 
2kr|x|^{2k-2}I. 
\end{equation*}
Since $V_0 $ has degree $j<2k,$ there is a constant $A>0$ such that $$-A(1+|x|^{2k-3})I\leq \frac{\partial^2 V_0(x)}{\partial x^2}\leq A(1+|x|^{2k-3})I.$$
 Therefore, we can estimate 
 \begin{equation}\label{example} \frac{\partial^2 V(x)}{\partial x^2}\geq 
\left(2kr|x|^{2k-2}-A|x|^{2k-3}-A\right)I.
\end{equation}
 We also observe that there exists  a positive constant $B$ such that 
$$-B(1+|x|^{2k-3})I\leq \frac{\partial^2(\partial_{x_{i}} V(x))}{\partial x^2}\leq B(1+|x|^{2k-3})I$$
for all $ i \in \{1,...,n\}.$ \eqref{example} shows that the smallest eigenvalue of $\frac{\partial^2 V(x)}{\partial x^2}$ satisfies $\alpha(x)\geq 2kr|x|^{2k-2}-A|x|^{2k-3}-A.$ Since $ 2kr|x|^{2k-2}-A|x|^{2k-3}-A$ grows faster than $B(1+|x|^{2k-3})$ as $|x| \to \infty,$ there are constants $c$ and $\tau\in [0,\nu)$ such that \eqref{Condition2} is satisfied. Thus, Theorem \ref{Main} applies to this type {of} potentials. In particular, it applies to double-well potentials of the form   $V(x)=r_1|x|^4-r_2|x|^2, \,\, \, r_1,r_2>0.$ 
}
\end{example}

\begin{Remark}
\begin{enumerate}
\item Our decay and regularization results above extend those of \cite{Herau}, where a stronger assumption, i.e. {$\partial^{2}_{x_i x_j}V \in \bigcap_{p=1}^{\infty}W^{p,\infty}(\mathbb{R}^n)$ }for all $i,j \in \{1,...,n\},$ was made. By contrast, we did not require the boundedness of the second and higher derivatives of $V.$
\item  Most of the previous works on the exponential convergence $f(t)\to f_{\infty}$ as $t\to \infty$  (e.g. \cite{Vil, DMSch, Baudoin, BGH, H^k, BS, CLW}) used the growth condition \eqref{Villanis condition} to get some weighted Poincar\'e type inequalities (see \cite[Lemma A.24]{Vil}), which are crucial in these works -- and additional to the Poincar\'e inequality \eqref{Poincare}. Our technique is rather different, based on  construction of appropriate state dependent matrices and state dependent matrix inequalities so that the (modified) dissipation functional (see \eqref{mod.dis.fun.}  below) decays exponentially.
\item Most of the previous methods for proving the exponential convergence do not give  an accurate decay rate, $\lambda$ is typically much too small there (see \cite[Section 7.2]{Vil}, \cite[Section 1.4]{DMSch}).   For example, in \cite[Section 7.2]{Vil}, the exponential decay rate $\lambda=\frac{1}{40}$ was obtained  for $V(x)=\frac{|x|^2}{2}$ and $\nu=\sigma=1.$  Since our decay rates are sharp  for quadratic potentials, in this setting, the true rate  $\lambda=\frac{1}{2}$ is given by Theorem \ref{Main} $(a)$  and $(e).$ 
\end{enumerate}
\end{Remark}
\section{Modified entropy methods for degenerate Fokker-Planck equations}
We first consider the following degenerate and non-symmetric  Fokker-Planck equation \cite{NFP, FAS}:
\begin{equation}\label{F-P}
\begin{cases}
\partial_t f=\text{div}(D\nabla f+(D+R)\nabla E f),  \,  \, \xi\in \mathbb{R}^d, \, t>0,\\
f(t=0)=f_0\in L_+^1(\mathbb{R}^d),\,\,  \int_{\mathbb{R}^d}f_0 \,d\xi=1
\end{cases}
\end{equation}
where $D\in \mathbb{R}^{d\times d}$ is a constant, symmetric, positive semi-definite $(\text{rank}(D)< d)$ matrix,
$R\in \mathbb{R}^{d\times d}$ is a constant skew-symmetric matrix.
$E:\mathbb{R}^d\to \mathbb{R}$ is a function which only depends on the state variable $\xi.$ We assume that $E$ is confining (i.e. $E(\xi)\to \infty$ for $|\xi|\to \infty$) and smooth enough so that \eqref{F-P} has a unique and smooth solution. 
 If $E$ grows fast enough, \eqref{F-P} has a normalized steady state $f_{\infty}=c_E e^{-E}, \, c_{E}>0. $
The weak maximum principle for degenerate parabolic equations \cite{WMP} can be applied to \eqref{F-P} and we can prove that $f(t,\xi)\geq 0$ for all $t>0, \, \xi \in \mathbb{R}^d.$ The divergence structure implies that  the initial mass is conserved and $f(t,\cdot)$ describes the evolution of a probability density
$$\int_{\mathbb{R}^d}f(t,\xi)d\xi=\int_{\mathbb{R}^d}f_0(\xi)d\xi=1, \, \, \, \forall t\geq 0.$$
We are interested in the large-time behavior of the solution, in particular, when  $\text{rank}(D)$ is  less than the dimension $d.$ 
When $D$ is positive definite (rank$(D)=d$), the large time behavior and exponential convergence have been studied comprehensively (see \cite{BE1}, \cite{OSI}, \cite{NFP}). One of the  well-know conditions which provides the exponential decay of the solution to the steady state is called  \textit{the Bakry-Emery condition} (see \eqref{BE} below) leading to:

\begin{theorem}[{\cite[Theorem 2.6]{NFP}}] Assume $\displaystyle \int_{\mathbb{R}^d}\left(\frac{f_0}{f_{\infty}}-1\right)^2 f_{\infty}d\xi<\infty$ and
 \begin{equation}\label{BE} \exists {\lambda}>0 \, \text{  such that   } \, \frac{\partial^2 E}{\partial \xi^2}(I+RD^{-1})+ \left(\frac{\partial^2 E}{\partial \xi^2}(I+RD^{-1})\right)^T\geq \lambda D^{-1},  \,   \,  \,\forall \xi \in \mathbb{R}^d. \end{equation}
 Then
 $$\int_{\mathbb{R}^d}\left(\frac{f(t)}{f_{\infty}}-1\right)^2 f_{\infty}d\xi\leq e^{-2\lambda t}\int_{\mathbb{R}^d}\left(\frac{f_0}{f_{\infty}}-1\right)^2 f_{\infty}d\xi.$$ 
 \end{theorem}
To prove the theorem above, one considers the  time derivative of the $L^2-$norm
 and we see that it decreases  
 \begin{equation}\label{t.der.L^2}\frac{d}{dt}\int_{\mathbb{R}^d}\left(\frac{f(t)}{f_{\infty}}-1\right)^2 f_{\infty}d\xi=-2\int_{\mathbb{R}^d}\nabla^T\left(\frac{f}{f_{\infty}}\right){D}\nabla\left(\frac{f}{f_{\infty}}\right)f_{\infty}d\xi =:-I(f(t)|f_{\infty})\leq 0.
 \end{equation}
$  I(f(t)|f_{\infty})$ is called the dissipation functional and since $D$ is positive definite it vanishes if and only if $f=f_{\infty}.$ It can be proven that, under the Bakry-Emery condition,
 \begin{equation}\label{diss.f}\frac{d}{dt} I(f(t)|f_{\infty})\leq -2\lambda I(f(t)|f_{\infty}).
 \end{equation}
Integrating this inequality from $(t,\infty)$ and using the convergences $I(f(t)|f_{\infty})\to 0$ and $\displaystyle \int_{\mathbb{R}^d}\left(\frac{f(t)}{f_{\infty}}-1\right)^2 f_{\infty}d\xi \to 0$ as $t \to \infty,$ it follows that \begin{equation}\label{Gron:ent}\frac{d}{dt}\int_{\mathbb{R}^d}\left(\frac{f(t)}{f_{\infty}}-1\right)^2 f_{\infty}d\xi\leq -2\lambda \int_{\mathbb{R}^d}\left(\frac{f(t)}{f_{\infty}}-1\right)^2 f_{\infty}d\xi
\end{equation} and, by Gr\"onwall's lemma, we get the desired result. 

When $D $ is only positive semi-definite, i.e. rank$(D)<d,$ one observes that $I(f(t)|f_{\infty})$ may vanish for certain probability densities $f\neq f_{\infty}.$ Hence the inequalities \eqref{diss.f} and \eqref{Gron:ent} will not hold in general. Since the above problems stem from the singularity of $D,$ one can  modify the dissipation function and define a modified  dissipation functional (see also \cite{FAS, AE})
\begin{equation}\label{mod.dis.fun.}S(f):=2\int_{\mathbb{R}^d}\nabla^T_{\xi}\left(\frac{f}{f_{\infty}}\right){P(\xi)}\nabla_{\xi}\left(\frac{f}{f_{\infty}}\right)f_{\infty}d\xi
\end{equation}
where $P:\mathbb{R}^d\to \mathbb{R}^{d \times d}$ is a symmetric positive definite matrix which will be chosen later. Extending the approach of \cite{FAS, AE}, we allow the matrix $P$ here to depend on $\xi\in \mathbb{R}^d.$   Our goal is to derive a  differential inequality similar to \eqref{diss.f} (like the dissipation functional satisfied for non-degenerate equations), i.e. \begin{equation}\label{IS}
\frac{d}{dt}S(f(t))\leq -2\lambda S(f(t)),\,  \, \, \, 
\end{equation}
for some $\lambda>0$ and a "good" choice of the matrix $P.$ If this holds true, we would obtain 
$$ S(f(t))\leq  S(f_0)e^{-2\lambda t}. $$  If  we can choose such  $P=P(\xi)\geq \eta I$ for some $\eta>0$ and all $\xi \in \mathbb{R}^d$, under the validity of the Poincar\'e inequality \eqref{Poincare} for $f_{\infty}(\xi)=c_E e^{-E(\xi)} $  (where $\begin{pmatrix}
    x\\ v
\end{pmatrix}$ in \eqref{Poincare} is replaced with $\xi$) we have
$$\int_{\mathbb{R}^d}\left(\frac{f(t)}{f_{\infty}}-1\right)^2 f_{\infty}d\xi\leq \frac{1}{C_{PI}} \int_{\mathbb{R}^{d}} \left|\nabla_{\xi} \left(\frac{f(t)}{f_{\infty}}\right)\right|^2f_{\infty}d\xi \leq \frac{1}{2C_{PI}\eta}S(f(t)), \, \, \, $$ which implies the exponential decay of the $L^2-$norm
$$\int_{\mathbb{R}^d}\left(\frac{f(t)}{f_{\infty}}-1\right)^2 f_{\infty}d\xi\leq \frac{1}{2 C_{PI}\eta}S(f_0) e^{-2\lambda t}.$$

 More generally, since the quadratic entropy is also a decreasing function of time $t$, instead of proving  \eqref{IS}, we can consider the functional 
 \begin{multline}\label{Phi1}
 \Phi(f(t))\colonequals\gamma \int_{\mathbb{R}^d}\left(\frac{f(t)}{f_{\infty}}-1\right)^2 f_{\infty}d\xi+S(f(t))\\= \gamma \int_{\mathbb{R}^d}\left(\frac{f(t)}{f_{\infty}}-1\right)^2 f_{\infty}d\xi+2\int_{\mathbb{R}^d}\nabla^T\left(\frac{f}{f_{\infty}}\right){P(\xi)}\nabla\left(\frac{f}{f_{\infty}}\right)f_{\infty}d\xi
 \end{multline} 
and choose a suitable parameter $\gamma\geq 0$ and a matrix $P$ such that 
\begin{equation}\label{dt Phi}
\frac{d \Phi(f(t))}{dt}\leq -2\lambda \Phi(f(t))\leq 0
\end{equation}
 for some $\lambda>0.$
{This} idea and method were successfully applied in \cite{AE} to \eqref{F-P} when the potential $E$ is quadratic.\\

We shall apply this method to the kinetic Fokker-Planck equation with non-quadratic $V(x).$ First, we denote 
$\xi\colonequals \begin{pmatrix}
x\\v
\end{pmatrix} \in \mathbb{R}^{2n},$ $E(\xi):=\frac{\nu}{\sigma}[V(x)+\frac{|v|^2}{2}],\, f_{\infty}=e^{-E}.$ Then the kinetic Fokker-Planck equation \eqref{KFP} can be written in the form of \eqref{F-P},
\begin{equation}\label{GFP}
\partial_t f=\text{div}_{\xi}(D\nabla_{\xi} f+(D+R)\nabla_{\xi} E f)
\end{equation}
with
\begin{equation}\label{D and R}
D=\begin{pmatrix}
0&0\\
0&\sigma I
\end{pmatrix} \in \mathbb{R}^{2n \times 2n} \, \, \, \, \, \, \text{and} \, \, \, \, \, \,   R=\frac{\sigma}{\nu}\begin{pmatrix}
0&-I\\
I&0
\end{pmatrix} \in \mathbb{R}^{2n \times 2n}.
\end{equation}
The rank of the diffusion matrix $D$ is $n<d=2n.$ Thus, \eqref{KFP} is both  non-symmetric and degenerate and the arguments above apply to the equation. 

 We will develop a modified entropy method. We will choose  {$\xi-$}dependent matrix $P$ in the modified dissipation functional  \eqref{mod.dis.fun.} so that \eqref{dt Phi} holds and  $\lambda>0$ is as large as possible.

We also mention that when the potential $E$ is quadratic in \eqref{F-P}, the question about the long time behavior can be reduced to an ODE problem:
\begin{theorem}\label{prop.norm}Let $0 \neq D \in \mathbb{R}^{d\times d} $ be positive semi-definite, $R\in \mathbb{R}^{d\times d}$ be skew-symmetric and $\mathbb{R}^{d} \ni \xi \to E(\xi)=\frac{{\xi}^T K^{-1} \xi}{2}$ for some positive definite matrix $K.$  Assume $(D+R)K^{-1}$ is positive stable and there is no  non-trivial subspace of $\text{\emph{Ker}}D$ which is invariant under $K^{-1}(D-R).$ If $f$ is the solution of \eqref{F-P} and $\xi(t)\in \mathbb{R}^d$ is the solution of the ODE $\dot{\xi}(t)=-K^{-\frac{1}{2}}(D+R)K^{-\frac{1}{2}}\xi$ with initial datum $\xi(0)=\xi_0,$ then 
\begin{equation}\label{prop.norm1}
\sup_{1\neq \frac{f_0}{f_{\infty}} \in L^2(\mathbb{R}^d, f_{\infty})}\frac{||f(t)/f_{\infty}-1||_{L^2(\mathbb{R}^d, f_{\infty})}}{||f_0/f_{\infty}-1||_{L^2(\mathbb{R}^d, f_{\infty})}}=\sup_{0\neq \xi_0\in \mathbb{R}^d}\frac{||\xi(t)||_2}{||\xi_0||_2}, \, \, \,  t\geq 0.
\end{equation}
\end{theorem} 
\begin{proof}
See \cite[Theorem 3.4]{prop}.
\end{proof} 
 One consequence of Theorem \ref{prop.norm} is that the decay estimate of the ODE-solution carries over to the corresponding Fokker-Planck equation.
\section{The choice of the matrix P}
For future reference (in the proof of Theorem \ref{Hyp.ellip.}) we shall now also allow the matrix $P$ to be time dependent. Hence we shall next consider the generalized functional
$$
  S(t,f):=2\int_{\mathbb{R}^d}\nabla^T_{\xi}\left(\frac{f}{f_{\infty}}\right){P(t,\xi)}\nabla_{\xi}\left(\frac{f}{f_{\infty}}\right)f_{\infty}d\xi .
$$
The following lemmas will play a crucial role in our arguments.
\begin{lemma}\label{lemma main} Let  $P:[0, \infty)\times \mathbb{R}^{2n}\to \mathbb{R}^{2n \times 2n}$ be smooth and $f$ be the solution of \eqref{KFP}, then 
\begin{multline}\label{derivS}
\frac{d}{dt} S(t,f(t))=-4\sigma\int_{\mathbb{R}^{2n}}\left\{\sum_{i=1}^n (\partial_{v_i}u)^TP\partial_{v_i}u\right\}f_{\infty}dxdv-2\int_{\mathbb{R}^{2n}}  u^T\left\{QP+PQ^T\right\}u f_{\infty}dxdv\\ -2\int_{\mathbb{R}^{2n}}  u^T\left\{[\nabla_x V\cdot \nabla_v-v\cdot \nabla_x+\nu v\cdot\nabla_v-\sigma \Delta_v-\partial_t]P\right\} u f_{\infty}dxdv,
\end{multline}
where $u\colonequals \nabla_{x,v} \left(\frac{f}{f_{\infty}}\right),$ $Q=Q(x)\colonequals \begin{pmatrix}
0&I\\
-\frac{\partial^2 V(x)}{\partial x^2}&\nu I
\end{pmatrix},$ and  $[\nabla_x V\cdot \nabla_v-v\cdot \nabla_x+\nu v\cdot\nabla_v-\sigma \Delta_v -\partial_t]$ denotes a scalar differential operator that is applied  to each element of the matrix $P=P(t,x,v)$. 
\end{lemma}
\begin{proof}
We denote $\displaystyle u_1\colonequals \nabla_x\left(\frac{f}{f_{\infty}}\right),$ $\displaystyle u_2\colonequals \nabla_v\left(\frac{f}{f_{\infty}}\right),$ then   $u_1$ and $u_2$ satisfy 
\begin{equation*}\partial_t u_1=\sigma \Delta_v u_1-\nu \sum_{i=1}^n v_i  \partial_{v_i}u_1+\sum_{i=1}^n \partial_{x_i} V \partial_{v_i} u_1-\sum_{i=1}^n v_i \partial_{x_i} u_1+\frac{\partial^2 V}{\partial x^2}u_2,
\end{equation*}
\begin{equation*}
 \partial_t u_2=\sigma \Delta_v u_2-\nu \sum_{i=1}^n v_i  \partial_{v_i}u_2+\sum_{i=1}^n \partial_{x_i} V \partial_{v_i} u_2-\sum_{i=1}^n v_i \partial_{x_i} u_2-u_1-\nu u_2.
 \end{equation*}
 These equations can be written  with respect to $u=\begin{pmatrix}
 u_1\\u_2
\end{pmatrix}:  $
 \begin{equation*}\partial_t u=\sigma \Delta_v u-\nu \sum_{i=1}^n v_i  \partial_{v_i}u+\sum_{i=1}^n \partial_{x_i} V \partial_{v_i} u-\sum_{i=1}^n v_i \partial_{x_i} u-Q^Tu.
\end{equation*}
It allows us to compute the time derivative of the modified dissipation functional
  \begin{multline}\label{time.derS}
 \frac{d}{dt} S(t,f(t))=4\int_{\mathbb{R}^{2n}}
  u^T P  \partial_t u
 f_{\infty}dxdv+2\int_{\mathbb{R}^{2n}}
  u^T  \partial_t P u
 f_{\infty}dxdv\\
= 4\sigma\int_{\mathbb{R}^{2n}} u^T P
 \Delta_v u  f_{\infty}dxdv-4\nu  \sum_{i=1}^n\int_{\mathbb{R}^{2n}} u^T P
   \partial_{v_i}u v_i f_{\infty}dxdv\\+ 4\sum_{i=1}^n \int_{\mathbb{R}^{2n}} u^TP
 \partial_{v_i} u \partial_{x_i} V f_{\infty}dxdv-4\sum_{i=1}^n\int_{\mathbb{R}^{2n}} u^TP
   \partial_{x_i}u v_i f_{\infty}dxdv\\-2\int_{\mathbb{R}^{2n}} u^T \{QP+PQ^T\}u f_{\infty}dxdv+2\int_{\mathbb{R}^{2n}}
  u^T  \partial_t P u
 f_{\infty}dxdv.
 \end{multline}
  First, we consider the term in the second line of \eqref{time.derS} and use $\partial_{v_i}f_{\infty}=-\frac{\nu}{\sigma}v_if_{\infty}:$
  \begin{multline}\label{term1}
   4\sigma \sum_{i=1}^n\int_{\mathbb{R}^{2n}} u^TP
  \partial^2_{v_iv_i} u f_{\infty}dxdv -4\nu  \sum_{i=1}^n\int_{\mathbb{R}^{2n}} u^TP
  \partial_{v_i}u  v_i f_{\infty}dxdv\\=-4\sigma \sum_{i=1}^n \int_{\mathbb{R}^{2n}}\partial_{v_i}u^TP\partial_{v_i}u f_{\infty}dxdv-4\sigma \sum_{i=1}^n\int_{\mathbb{R}^{2n}} u^T (\partial_{v_i}P)\partial_{v_i}u f_{\infty}dxdv.
 \end{multline}
 By integrating by parts the last term of \eqref{term1} we obtain
  \begin{multline*}
  -4\sigma \sum_{i=1}^n\int_{\mathbb{R}^{2n}} u^T (\partial_{v_i}P)\partial_{v_i}u f_{\infty}dxdv\\= 4\sigma \sum_{i=1}^n\int_{\mathbb{R}^{2n}} u^T (\partial_{v_i}P)\partial_{v_i}u f_{\infty}dxdv+4\sigma \sum_{i=1}^n\int_{\mathbb{R}^{2n}} u^T (\partial^2_{v_i v_i}P)u f_{\infty}dxdv-4\nu \sum_{i=1}^n\int_{\mathbb{R}^{2n}} u^T (\partial_{ v_i}P)u v_i f_{\infty}dxdv
  \end{multline*}
  and we find
  \begin{multline*}
  -4\sigma \sum_{i=1}^n\int_{\mathbb{R}^{2n}} u^T (\partial_{v_i}P)\partial_{v_i}u f_{\infty}dxdv=2\sigma \int_{\mathbb{R}^{2n}} u^T (\Delta_v P)u f_{\infty}dxdv-2\nu \sum_{i=1}^n\int_{\mathbb{R}^{2n}} u^T (v_i \partial_{ v_i}P)u  f_{\infty}dxdv.
  \end{multline*}
 If we use this equality  in \eqref{term1}, we get \begin{multline}\label{term11}
 4\sigma\int_{\mathbb{R}^{2n}} u^T P
 \Delta_v u  f_{\infty}dxdv-4\nu  \sum_{i=1}^n\int_{\mathbb{R}^{2n}} u^T P
 v_i  \partial_{v_i}u f_{\infty}dxdv\\=-4\sigma \sum_{i=1}^n \int_{\mathbb{R}^{2n}}(\partial_{v_i}u)^TP\partial_{v_i}u f_{\infty}dxdv-2\int_{\mathbb{R}^{2n}} u^T \{[\nu v\cdot\nabla_v -\sigma \Delta_v] P\}u f_{\infty}dxdv.
 \end{multline}
  Next, we integrate by parts in the terms in the third line of \eqref{time.derS}: \begin{multline}\label{third term1}
  4\sum_{i=1}^n \int_{\mathbb{R}^{2n}} u^TP
 \partial_{v_i} u \partial_{x_i}V f_{\infty}dxdv\\= -4\sum_{i=1}^n \int_{\mathbb{R}^{2n}} u^TP
 \partial_{v_i} u \partial_{x_i} V f_{\infty}dxdv-4\sum_{i=1}^n \int_{\mathbb{R}^{2n}} u^T(\partial_{v_i} P)
 u  \partial_{x_i} V  f_{\infty}dxdv+\frac{4\nu}{\sigma}\sum_{i=1}^n \int_{\mathbb{R}^{2n}} u^TP
 u  \partial_{x_i} V v_i f_{\infty}dxdv,
 \end{multline}
 \begin{multline}\label{third term2}
 -4\sum_{i=1}^n\int_{\mathbb{R}^{2n}} u^TP
 \partial_{x_i}u  v_i  f_{\infty}dxdv\\=4\sum_{i=1}^n\int_{\mathbb{R}^{2n}} u^TP
   \partial_{x_i}u v_i f_{\infty}dxdv+4\sum_{i=1}^n\int_{\mathbb{R}^{2n}} u^T(\partial_{x_i}P)
   uv_i f_{\infty}dxdv-\frac{4\nu}{\sigma}\sum_{i=1}^n\int_{\mathbb{R}^{2n}} u^TP
   u \partial_{x_i}V v_i f_{\infty}dxdv.
 \end{multline}
 \eqref{third term1} and \eqref{third term2} show that the third line of  \eqref{time.derS} equals
 \begin{multline}\label{thirdterm}
  -2\sum_{i=1}^n \int_{\mathbb{R}^{2n}} u^T(\partial_{v_i} P)
  u \partial_{x_i} V f_{\infty}dxdv+2\sum_{i=1}^n\int_{\mathbb{R}^{2n}} u^T(\partial_{x_i}P)
   u v_i f_{\infty}dxdv\\=-2\int_{\mathbb{R}^{2n}} u^T\{[\nabla_x V\cdot \nabla_v-v\cdot \nabla_x]P\}
  u   f_{\infty}dxdv.
 \end{multline}
 Combining \eqref{time.derS}, \eqref{term11}, and \eqref{thirdterm} we obtain the statement \eqref{derivS}.
\end{proof}

\begin{Remark}\label{S-Gamma1}We give now a (formal)  generalization of  the above result \eqref{derivS} to Markovian evolution equations using the \emph{Gamma calculus}, see, e.g., \cite{AGMar, Baudoin, BGH}:\\
\indent
First, let $L$ be the generator of some Markovian evolution on $\R^d$ with corresponding invariant measure $f_\infty d\xi$.
Let $P=P(\xi)$ be a smooth matrix function (but it does not have to be symmetric or positive definite). We define the first order bilinear form
$$
  \Gamma^P(g,h)\colonequals \nabla_\xi g^T P \nabla_\xi h
$$
and 
$$
  \Gamma^P_2(g,h)\colonequals \frac12 \big(L \Gamma^P(g,h)-\Gamma^P(Lg,h)-\Gamma^P(g,Lh)\big)\,.
$$
For a solution $h(t)$ of $\partial_t h=Lh$, these definitions give 
\begin{equation}\label{Gamma-der1}
  \frac{d}{dt} \Gamma^P(h,h)=\Gamma^P(Lh,h)+\Gamma^P(h,Lh)=-2 \Gamma_2^P(h,h) + L \Gamma^P(h,h),\qquad\forall \xi\in\R^d.
\end{equation}
 We use $ \Gamma^P$ to define the modified dissipation functional
$$
  S(f) \colonequals  2\int_{\R^d} \Gamma^P(h,h)\, f_\infty d\xi\quad\text{ with } h=\frac{f}{f_\infty}.
$$
 We obtain by integrating \eqref{Gamma-der1}:
\begin{equation}\label{s-der1}
  \frac{d}{dt} S(f(t)) = -4\int_{\R^d} \Gamma^P_2(h,h)\, f_\infty d\xi \,,
\end{equation}
where we used that $\int_{\mathbb{R}^d} L\Gamma^P(h,h)\, f_\infty d\xi=0$.\\
\indent
In particular, let $L$ be the generator of the kinetic Fokker-Planck equation \eqref{KFP2}, and we recall that $\xi\colonequals \begin{pmatrix}x\\v\end{pmatrix} .$ 
Then, a straightforward (but lengthy) computation shows that 
\begin{multline*}\label{derivS21}
2\Gamma_2^P(h,h)=2\sigma\sum_{i=1}^n (\partial_{v_i}u)^TP\partial_{v_i}u + u^T\left(QP+PQ^T\right)u + u^T (LP) u  
\\+2\sigma \sum_{i=1}^n (\partial_{v_i}u)^T(\partial_{v_i}P)u +4\sigma \sum_{i=1}^n u^T(\partial_{v_i}P)\partial_{v_i}u.
\end{multline*}
One can check (by integrating by parts the term $4 \sigma\int_{\R^d} \sum_{i=1}^n u^T(\partial_{v_i}P)\partial_{v_i} u  f_\infty d\xi$ in the right  hand side of \eqref{s-der1}) that  \eqref{s-der1} coincides with \eqref{derivS}. Hence, \eqref{s-der1} reproduces \eqref{derivS}. But in contrast to \eqref{derivS}, the preceding statement \eqref{Gamma-der1} is local in $\xi$ and therefore stronger.
%
\end{Remark}
The key question for using the modified entropy dissipation functional $S (f )$ is how to choose the matrix $P.$ To determine $P$ we
shall need the following algebraic result:
\begin{lemma}\label{mat.ine}
For any fixed matrix $\mathcal{Q}\in \mathbb{R}^{d\times d},$ let $\mu:=\min\{\textsf{Re}(\beta): \beta \text{ is an eigenvalue of }  \mathcal{Q}\}.$ Let $\{\beta_m: 1\leq m\leq m_0\}$ be all the eigenvalues of $\mathcal{Q}$ with $\mu=\textsf{Re}(\beta),$ only counting their geometric multiplicity.

(a) If $\beta_m$ is non-defective for all $m\in \{1,...,m_0\},$ then there exists a symmetric, positive definite matrix $P\in \mathbb{R}^{d\times d}$ with 
$$\mathcal{Q} P+P\mathcal{Q}^T\geq 2\mu P.$$ 

(b) If $\beta_m$ is defective for at least one  $m\in \{1,...,m_0\},$ then for any $\varepsilon>0$ there exists a symmetric, positive definite matrix $P(\varepsilon)\in \mathbb{R}^{d\times d}$ with 
$$\mathcal{Q}P(\varepsilon)+P(\varepsilon)\mathcal{Q}^T\geq 2(\mu-\varepsilon) P(\varepsilon).$$ 
\end{lemma}
\begin{proof}
See \cite[Lemma 4.3]{AE}.
\end{proof}

 We consider the matrix function \begin{equation}\label{Q(x)}
 Q(x)\colonequals \begin{pmatrix}
0&I\\
-\frac{\partial^2 V(x)}{\partial x^2}&\nu I
\end{pmatrix}, \, \, \, \, x \in \mathbb{R}^n,
\end{equation}
which appears in \eqref{derivS}. We want to 
construct a symmetric positive definite matrix $P(x)$ such that $Q(x)P(x)+P(x)Q^T(x)$ is positive definite and 
$$Q(x)P(x)+P(x)Q^T(x)\geq 2\mu  P(x)$$ 
for some $\mu>0$  and for all $x \in \mathbb{R}^n.$ 
We recall  $$\displaystyle \alpha(x)\colonequals \min_{i\in \{1,..,n\} }\left\{\alpha_i(x)\, : \,\alpha_i (x) \, \text{ is an eigenvalue of } \frac{\partial^2 V(x)}{\partial x^2}\right\},$$ $$\displaystyle \alpha_0 \colonequals\inf_{x \in \mathbb{R}^n}\alpha(x),$$
$$\mu\colonequals \inf_{x \in \mathbb{R}^n, \,i\in \{1,..,n\}}\{\textsf{Re}(\beta_i(x)): \beta_i(x) \text{ is an eigenvalue of }  Q(x)\}.$$

\begin{lemma}\label{mat.ine1} 1)  The matrix $Q(x)$ is positive stable at any fixed $x \in \mathbb{R}^n,$ if and only if $\frac{\partial^2 V(x)}{\partial x^2}$ is positive definite.\\
2) Let $\frac{\partial^2 V(x)}{\partial x^2}$ be positive definite for some $x \in \mathbb{R}^n.$ Then:
 \begin{enumerate} \item[(a)] If $\alpha_0> \frac{\nu^2}{4},$ then $\mu=\frac{\nu}{2}$ and there exists a symmetric positive definite matrix $P(x)$ such that 
$$Q(x)P(x)+P(x)Q^T(x)=2\mu P(x).$$
\item[(b)] If $0<\alpha_0<  \frac{\nu^2}{4},$ then  $\mu=\frac{\nu-\sqrt{\nu^2-4\alpha_0}}{2}$ and there exists a symmetric positive definite matrix $P(x)$ such that 
$$Q(x)P(x)+P(x)Q^T(x)\geq 2\mu P(x).$$
\item[(c)] If $\alpha_0 =\frac{\nu^2}{4},$ then $\mu=\frac{\nu}{2}$  and, for any $\varepsilon\in (0, \nu),$  there exists a symmetric positive definite matrix $P(x,\varepsilon)$ such that 
$$Q(x)P(x,\varepsilon)+P(x,\varepsilon)Q^T(x)\geq (2\mu-\varepsilon) P(x,\varepsilon).$$ 
\end{enumerate}
\end{lemma}

\begin{proof}
{Part $1)$}
Let $x$ be any point of $\mathbb{R}^n,$ we compute the eigenvalues $\beta(x)$ of $Q(x).$ If $\beta(x)\neq 0$ we have {the condition}
\begin{multline*}
\det(Q(x)-\beta(x) I)=\begin{vmatrix}
-\beta(x) I&I\\
-\frac{\partial^2 V(x)}{\partial x^2}&(\nu-\beta(x)) I
\end{vmatrix}\\=\frac{1}{(\beta(x))^n} \begin{vmatrix}
-\beta(x) I&0\\
-\frac{\partial^2 V(x)}{\partial x^2}&-\frac{\partial^2 V(x)}{\partial x^2}+\beta(x)(\nu-\beta(x))I 
\end{vmatrix}\\={(-1)^n }\det\left(-\frac{\partial^2 V(x)}{\partial x^2}+\beta(x)(\nu-\beta(x))I\right)=0.
\end{multline*}
Let $\alpha_i(x)\in \mathbb{R}, \, \, i \in \{1,...,n\}$ denote the eigenvalues of $\frac{\partial^2 V(x)}{\partial x^2},$ then
the above eigenvalue condition reads $$ \prod_{i=1}^n(\beta^2(x)-\nu \beta(x)+\alpha_i(x))=0.$$ 
Hence  the non-zero eigenvalues of $Q(x)$ are
\begin{equation}\label{eigenv}
\beta^{\pm}_i(x)=\begin{cases} \frac{\nu \pm \sqrt{{\nu}^2-4 \alpha_i(x)}}{2}, \, \, \, \,\, \,  \text{if }\, \, \, \, \nu^2 \geq 4 \alpha_i(x) \\
\frac{\nu\pm  \mathbf{i} \sqrt{4 \alpha_i(x)-{\nu}^2}}{2}, \, \, \, \, \, \text{if }\, \, \, \, \nu^2 < 4 \alpha_i(x)
\end{cases}, 
\, \, i \in \{1,...,n\},
\end{equation}
where $\mathbf{i}=\sqrt{-1}.$ 
 Moreover,  $\beta(x)=0$ can be an eigenvalue of $Q(x)$ iff one of the eigenvalues of $\frac{\partial^2 V(x)}{\partial x^2}$ is zero. 
This  shows that $Q(x)$ is positive stable (i.e.,\,the eigenvalues $\beta_{i}(x)$ have positive  real part) iff $\frac{\partial^2 V(x)}{\partial x^2}> 0. $\\ 

For Part $2)$  we shall construct  matrices $P(x),$ which relies on the proof of Lemma \ref{mat.ine} (Lemma 4.3 in \cite{AE}).

$(a)$ Let $\alpha_0> \frac{\nu^2}{4}.$
 In this case, because of \eqref{eigenv} the matrix $Q(x)$ is positive stable  and $\mu=\frac{\nu}{2}>0.$
We define the matrix 
    $$P(x)\colonequals \begin{pmatrix}
2I&\nu I\\
\nu I& 2\frac{\partial^2 V(x)}{\partial x^2}
\end{pmatrix},$$
and for this choice, it is easy to check that $$Q(x)P(x)+P(x)Q^T(x)=\nu P(x)=2\mu P(x).$$
To make sure that $P(x)$ is positive definite, we compute the eigenvalues $\eta(x)$ of $P(x)$ at each $x \in \mathbb{R}^n:$ {For $\eta(x)\neq 2$ we have the condition 
\begin{multline*}
\det(P(x)-\eta(x) I)=\begin{vmatrix}
(2-\eta(x)) I&\nu I\\
\nu I& 2\frac{\partial^2 V(x)}{\partial x^2}-\eta(x) I
\end{vmatrix}\\=\frac{1}{(2-\eta(x))^n} \begin{vmatrix}
(2-\eta(x)) I&0\\
\nu I&(2-\eta(x))\left(2\frac{\partial^2 V(x)}{\partial x^2}-\eta(x) I\right)-\nu^2 I 
\end{vmatrix}\\= \det\left((2-\eta(x))\left(2\frac{\partial^2 V(x)}{\partial x^2}-\eta(x) I\right)-\nu^2 I\right)=0.
\end{multline*}
 $\eta(x)=2$ is not an eigenvalue of $P(x)$ and so the eigenvalues of $P(x)$ satisfy
$$\prod_{i=1}^n\left(\eta^2(x)-(2+2\alpha_i(x))\eta(x)+4\alpha_i(x)-\nu^2\right)=0.$$}
 We conclude that the eigenvalues are 
$$\eta^{\pm}_i(x)=1+\alpha_i(x)\pm \sqrt{(\alpha_i(x)+1)^2-(4\alpha_i(x)-\nu^2)}, \, \, i \in \{1,...,n\}.$$
Since we assumed  $\alpha_i(x)\geq \alpha(x)\geq \alpha_0> \frac{\nu^2}{4}$ for all $i\in \{1,...,n\},$  the eigenvalues are positive and satisfy  $$\eta\colonequals \inf_{x \in \mathbb{R}^n, \,i\in\{1,...,n\}}{\eta^{\pm}_i(x)}=1+\alpha_0- \sqrt{(\alpha_0+1)^2-(4\alpha_0-\nu^2)}>0.$$ Thus, $P(x)$ is positive definite and $P(x)\geq \eta I$ for all $x \in \mathbb{R}^n.$  

$(b)-(c)$ Let  $0<\alpha_0\leq\frac{\nu^2}{4}.$ Then \eqref{eigenv} shows $\mu=\frac{\nu-\sqrt{\nu^2-4\alpha_0}}{2}.$ Let $\varepsilon>0$ be a fixed small number. We define $$\omega\colonequals\begin{cases}\alpha_0, \, \, \text{ if } \alpha_0<\frac{\nu^2}{4}\\
\alpha_0-\frac{\varepsilon^2}{4},  \, \, \text{ if } \alpha_0=\frac{\nu^2}{4}\end{cases} $$  
and consider the matrix $$P(x)\colonequals \begin{pmatrix}
2I&\nu I\\
\nu I& 2\frac{\partial^2 V(x)}{\partial x^2}+(\nu^2-4\omega)I
\end{pmatrix}.$$
We compute its eigenvalues $\eta(x)$ by a similar computation as above: 
\begin{equation}\label{eig.val.P}\eta^{\pm}_i(x)=1+\zeta_i(x)\pm \sqrt{(\zeta_i(x)+1)^2-(4\zeta_i(x)-\nu^2)},
\end{equation}
where $\zeta_i(x)\colonequals \alpha_i(x)+\frac{\nu^2}{2}-2\omega>\frac{\nu^2}{4}.$ We also have
{ $$\eta:=\inf_{x \in \mathbb{R}^n, \,i\in\{1,...,n\}}{\eta^{\pm}_i(x)}={
1+\alpha_0+\frac{\nu^2}{2}-2\omega- \sqrt{(\alpha_0+\frac{\nu^2}{2}-2\omega-1)^2+\nu^2}} >0.$$} Thus, $P(x)$ is positive definite and $P(x)\geq \eta I$ for all $x \in \mathbb{R}^n.$
Then we compute
 \begin{multline}\label{QP+PQ..}
 Q(x)P(x)+P(x)Q^T(x)
\\=
(\nu-\sqrt{\nu^2-4\omega}) P(x)+\sqrt{\nu^2-4\omega}\begin{pmatrix}
2 I&(\nu +\sqrt{\nu^2-4\omega}) I\\
(\nu+\sqrt{\nu^2-4\omega})I& 2\frac{\partial^2 V}{\partial x^2}+\sqrt{\nu^2-4\omega}(\nu+\sqrt{\nu^2-4\omega})  I
\end{pmatrix}.
\end{multline}
Since $\frac{\partial^2 V}{\partial x^2}\geq \omega I,$ {the second matrix in the last line of \eqref{QP+PQ..} is bounded below by 
\begin{multline*}\label{m.a}
 \begin{pmatrix}
2 I&(\nu +\sqrt{\nu^2-4\omega}) I\\
(\nu+\sqrt{\nu^2-4\omega})I& 2\omega+\sqrt{\nu^2-4\omega}(\nu+\sqrt{\nu^2-4\omega})  I
\end{pmatrix}\\ =\begin{pmatrix}
2 I&(\nu +\sqrt{\nu^2-4\omega}) I\\
(\nu+\sqrt{\nu^2-4\omega})I& \frac{1}{2}(\nu+\sqrt{\nu^2-4\omega})^2  I
\end{pmatrix}\geq  0. 
\end{multline*}}

Consequently, we get 
\begin{equation*}
Q(x)P(x)+P(x)Q^T(x)\geq (\nu-\sqrt{\nu^2-4\omega})P(x) \, \, \, \text{ for all }\, \, \, x \in \mathbb{R}^n.
\end{equation*}
\end{proof}

Lemma \ref{mat.ine1} shows that, if $\frac{\partial^2 V(x)}{\partial x^2} $ is not positive definite at some $x \in \mathbb{R}^n$ (and hence $\alpha_0\leq 0$), then $Q(x)$ is not positive stable. In this case,  it is not possible to  find a positive constant $\mu$ and a positive definite matrix $P(x)$ such that  $Q(x)P(x)+P(x)Q^T(x)\geq \mu P(x).$ If $\alpha_0$ is just finite and not necessarily positive, we have the following modified inequality.
\begin{lemma}\label{gen.mat.inq}
Let $\alpha_0>-\infty.$ Then there exist $\gamma\geq 0,$ $\delta\in [0, \nu),$ and a symmetric positive definite matrix function $P(x)$ such that
\begin{equation}\label{aux,mat.ineq}
Q(x)P(x)+P(x)Q^T(x)+\gamma D\geq (\nu-\delta)P(x), \, \, \, \, \forall x \in \mathbb{R}^d,
\end{equation}
where  $D=\begin{pmatrix}0&0\\0&\sigma I\end{pmatrix} \in \mathbb{R}^{2n \times 2n}$ is the matrix defined in \eqref{GFP}. 
\end{lemma} 
\begin{proof}
Let  $a\geq 0$ be any constant such that  $a+\alpha_0>\frac{\nu^2}{4}.$ 
We consider the matrix 
\begin{equation*}
P(x)\colonequals \begin{pmatrix}
2I&\nu I\\
\nu I& 2\frac{\partial^2 V(x)}{\partial x^2}+2aI
\end{pmatrix}.
\end{equation*}
{In analogy to \eqref{eig.val.P} we find its eigenvalues as } $$\eta^{\pm}_i(x)=1+\zeta_i(x)\pm \sqrt{(\zeta_i(x)+1)^2-(4\zeta_i(x)-\nu^2)},$$
where $\zeta_i(x):=\alpha_i(x)+a\geq a+\alpha_0>\frac{\nu^2}{4},$ and  $\alpha_i(x)\in \mathbb{R}, \, \, i \in \{1,...,n\}$ denote the eigenvalues of $\frac{\partial^2 V(x)}{\partial x^2}.$
We also have
 \begin{equation}\label{eta}
 \eta\colonequals \inf_{x\in \mathbb{R}^n, i\in \{1,...,n\}}{\eta^{\pm}_i(x)}= \frac{4( a+\alpha_0-\frac{\nu^2}{4})}{
1+ a+\alpha_0+ \sqrt{( a+\alpha_0-1)^2+\nu^2}} >0.
\end{equation}   Thus, $P(x)$ is uniformly positive definite and $P(x)\geq \eta I$ for all $x \in \mathbb{R}^n.$

 Next we compute
 $$QP+PQ^T+\gamma D=
\nu P+\begin{pmatrix}
0&2a I\\
2aI& (2\nu a+\gamma \sigma)I
\end{pmatrix}$$ \begin{equation}\label{QP+PQ+yD}
=(\nu-\delta) P+\begin{pmatrix}
2\delta I&(\nu \delta +2a) I\\
(\nu\delta+2a)I& \delta(2\frac{\partial^2 V}{\partial x^2}+2aI)+(2\nu a+\gamma \sigma)I
\end{pmatrix},
\end{equation}
where { $\delta\in  [0, \nu)$  will be chosen later.}
We compute the (real) eigenvalues $\theta$ of the symmetric matrix 
\begin{equation}\label{matrix.aux}
\begin{pmatrix}
2\delta I&(\nu \delta +2a) I\\
(\nu\delta+2a)I& \delta(2\frac{\partial^2 V}{\partial x^2}+2aI)+(2\nu a+\gamma \sigma)I
\end{pmatrix} 
\end{equation} which appears in \eqref{QP+PQ+yD}:\\
{For $\theta(x)\neq 2 \delta$  we have the condition}
\begin{multline*}\begin{vmatrix}
(2\delta-\theta) I&(\nu \delta +2a) I\\
(\nu\delta+2a)I& \delta(2\frac{\partial^2 V}{\partial x^2}+2aI)+(2\nu a+\gamma \sigma-\theta)I
\end{vmatrix}\\= \frac{1}{(2\delta-\theta)^n}\begin{vmatrix}
(2\delta-\theta) I& 0 \\
(\nu \delta+2a)I & (2\delta-\theta) \left(\delta(2\frac{\partial^2 V}{\partial x^2}+2aI)+(2\nu a+\gamma \sigma-\theta)I \right)-(\nu \delta +2a)^2 I 
\end{vmatrix}\\  
=\bigg| (2\delta-\theta) \left(\delta(2\frac{\partial^2 V}{\partial x^2}+2aI)+(2\nu a+\gamma \sigma-\theta)I \right)-(\nu \delta +2a)^2 I \bigg| \\
=\prod_{i=1}^n \left(\theta^2-\theta\left[2\delta(\alpha_i(x)+a)+2\delta+2\nu a+\gamma\sigma\right]+4\delta^2(\alpha_i(x)+a-{\nu^2}/{4})+2\delta \gamma \sigma -4a^2 \right)=0.
\end{multline*}
Let us consider the following equations with $i \in \{1,...,n\}:$
\begin{equation}\label{eign.aux}\theta^2-\theta[2\delta(\alpha_i(x)+a)+2\delta+2\nu a+\gamma\sigma]+[4\delta^2(\alpha_i(x)+a-{\nu^2}/{4})+2\delta \gamma \sigma -4a^2]=0,
\end{equation} and we shall show that they have non-negative solutions for an appropriate choice of $\delta$ and $\gamma.$ To this end  we see first that 
$$2\delta(\alpha_i(x)+a)+2\delta+2\nu a+\gamma\sigma\geq 2\delta(\alpha_0+a)+2\delta\geq \frac{\delta \nu^2}{2}+2\delta\geq 0.$$
Next, we choose  
\begin{equation}\label{delta}
\displaystyle \delta=\delta(a,\gamma)\colonequals \frac{1}{\sqrt{a+\alpha_0-\frac{\nu^2}{4}}}\left[\sqrt{\left(\frac{\gamma \sigma}{4\sqrt{a+\alpha_0-\frac{\nu^2}{4}}}\right)^2+a^2}-\frac{\gamma \sigma}{4\sqrt{a+\alpha_0-\frac{\nu^2}{4}}}\right]\geq 0,
\end{equation} which satisfies    \begin{equation}\label{quad.eq}
4\delta^2(a+\alpha_0-\frac{\nu^2}{4}) +2\delta \gamma \sigma -4a^2=0.
\end{equation} 
Hence, the last term of \eqref{eign.aux} satisfies
$$4\delta^2(\alpha_i(x)+a-\frac{\nu^2}{4})+2\delta \gamma \sigma -4a^2\geq 4\delta^2(a+\alpha_0-\frac{\nu^2}{4}) +2\delta \gamma \sigma -4a^2=0 $$
for all $i \in \{1,...,n\}.$
Therefore, the quadratic equations \eqref{eign.aux} have non-negative coefficients and so their solutions, i.e.  the eigenvalues of \eqref{matrix.aux}, are non-negative. Consequently, we get \eqref{aux,mat.ineq}.

We note that $\delta$ from \eqref{delta} satisfies, for any fixed $a>\frac{\nu}{4}-\alpha_0,$ $\delta(a, \gamma) \to 0$ as $\gamma \to \infty.$ Hence, choosing $\gamma$ large enough, we have $\delta\in [0,\nu).$ 
\end{proof}
\begin{Remark} 
If $\alpha_0>0,$ we can take $\gamma=0$ in Lemma \ref{gen.mat.inq}. This follows by choosing in the proof of Lemma \ref{gen.mat.inq}
$$a=\begin{cases} 
0,& \text{ if } \alpha_0>\frac{\nu^2}{4}\\
 \frac{\varepsilon^2}{2},& \text{ if } \alpha_0=\frac{\nu^2}{4} \\
 \frac{\nu^2-4\alpha_0}{2},& \text{ if } 0<\alpha_0<\frac{\nu^2}{4} \end{cases}, \, \, \, \, \delta=\begin{cases} 
0,& \text{ if } \alpha_0>\frac{\nu^2}{4}\\
 \frac{\varepsilon}{\sqrt{2}},& \text{ if } \alpha_0=\frac{\nu^2}{4} \\
 \sqrt{\nu^2-4\alpha_0},& \text{ if } 0<\alpha_0<\frac{\nu^2}{4} \end{cases},$$ with any $\varepsilon\in (0, \nu).$  Therefore, Lemma \ref{gen.mat.inq} includes the second part of Lemma \ref{mat.ine1}. 
 However, if $\alpha_0\leq 0,$ we have to choose $\gamma>0.$  
\end{Remark}
\section{Proofs }
\subsection{Proof of Theorem \ref{Main} }
\begin{proof}
We denote 
 $\displaystyle u_1\colonequals \nabla_{x}\left(\frac{f}{f_{\infty}}\right),\, \,  u_2\colonequals \nabla_{v}\left(\frac{f}{f_{\infty}}\right),$ and $  u\colonequals \begin{pmatrix}
u_1\\u_2
\end{pmatrix}.$
We consider the modified  dissipation functional
$$S(f(t))=2 \int_{\mathbb{R}^{2n}}u^T(t) Pu(t) f_{\infty}dxdv$$
for some {symmetric positive definite} matrix $P=P(x,v)\in \mathbb{R}^{2n \times 2n}.$ By Lemma \ref{lemma main} (for a $t$-independent matrix $P$) we have  
\begin{multline}\label{SKFP}\displaystyle\frac{d}{dt} S(f(t))=-4\sigma\int_{\mathbb{R}^{2n}}\left\{\sum_{i=1}^n (\partial_{v_i}u)^TP\partial_{v_i}u\right\}f_{\infty}dxdv-2\int_{\mathbb{R}^{2n}}  u^T\left\{QP+PQ^T\right\}u f_{\infty}dxdv\\ -2\int_{\mathbb{R}^{2n}}  u^T\left\{[\nabla_x V\cdot \nabla_v-v\cdot \nabla_x+\nu v\cdot\nabla_v-\sigma \Delta_v]P\right\} u f_{\infty}dxdv,
\end{multline}
with $Q(x)= \begin{pmatrix}
0&I\\
-\frac{\partial^2 V(x)}{\partial x^2}&\nu I
\end{pmatrix}.$
Let $c\in \mathbb{R}$ and $\tau\in [0,\nu) $ are the  constants such that Assumption \ref{A1} is satisfied. Since \eqref{Condition1} is positive semi-definite,
$\frac{\partial^2  V(x)}{\partial x^2}+cI$
is also  positive semi-definite and so $\frac{\partial^2  V(x)}{\partial x^2}\geq -cI$ for all $x \in \mathbb{R}^n.$  We define the matrix $P$ depending on the constant $c.$
\subsubsection*{Case $(a)$:}Assume $c\leq -\frac{\nu^2}{4}, \,\, \alpha_0>\frac{\nu^2}{4}.$ 
By Lemma \ref{mat.ine1}  $(2a)$ and {by its proof},  the matrix  $P(x)\colonequals \begin{pmatrix}
2I&\nu I\\
\nu I& 2\frac{\partial^2 V(x)}{\partial x^2}
\end{pmatrix}$
satisfies $$Q(x)P(x)+P(x)Q^T(x)=\nu P(x)\, \,\text{ and } \, \,P(x)\geq \eta I $$ 
for all $ x \in \mathbb{R}^n$ and $\eta\colonequals 1+\alpha_0- \sqrt{(\alpha_0+1)^2-(4\alpha_0-\nu^2)}>0.$
For this choice of the matrix $P,$ 
\begin{equation}\label{[]P}
[\nabla_x V\cdot \nabla_v-v\cdot \nabla_x+\nu v\cdot\nabla_v-\sigma \Delta_v]P(x)=\begin{pmatrix}
0&0\\
0& -2\frac{\partial^2 (v\cdot \nabla_x V)}{\partial x^2}
\end{pmatrix}.
\end{equation}
Then \eqref{SKFP} can be written as
\begin{multline}\label{SKFP1}
\frac{d}{dt} S(f(t))=-4\sigma\int_{\mathbb{R}^{2n}}\left\{\sum_{i=1}^n (\partial_{v_i}u)^TP\partial_{v_i}u\right\}f_{\infty}dxdv\\ -2\nu\int_{\mathbb{R}^{2n}}  u^T Pu f_{\infty}dxdv +4\int_{\mathbb{R}^{2n}} u^T\begin{pmatrix}
0&0\\
0& \frac{\partial^2 (v\cdot \nabla_x V)}{\partial x^2}
\end{pmatrix} u f_{\infty}dxdv\\=
-4\sigma\int_{\mathbb{R}^{2n}}\left\{\sum_{i=1}^n (\partial_{v_i}u)^TP\partial_{v_i}u\right\}f_{\infty}dxdv-\nu S(f(t))+4\int_{\mathbb{R}^{2n}} u^T\begin{pmatrix}
0&0\\
0& \frac{\partial^2 (v\cdot \nabla_x V)}{\partial x^2}
\end{pmatrix} u f_{\infty}dxdv.
\end{multline}
We shall now consider each term of this equation. First we compute
\begin{multline}\label{est:S}
S(f(t))=2\int_{\mathbb{R}^{2n}}\left\{2 |u_{1}|^2+2\nu  u_{1}\cdot u_{2}+2u^T_{2}\frac{\partial^2  V}{\partial x^2}u_{2}\right\}f_{\infty}dxdv\\=4\int_{\mathbb{R}^{2n}}|u_{1}+\frac{\nu}{2} u_{2}|^2f_{\infty}dxdv+4\int_{\mathbb{R}^{2n}} u^T_{2}\left(\frac{\partial^2  V}{\partial x^2}-\frac{\nu^2}{4}I\right)u_{2} f_{\infty}dxdv\\ \geq 4\int_{\mathbb{R}^{2n}} u^T_{2}\left(\frac{\partial^2  V}{\partial x^2}-\frac{\nu^2}{4}I\right)u_{2} f_{\infty}dxdv.
 \end{multline}
  Then 
\begin{multline}\label{est:Tr}\displaystyle 4\sigma\int_{\mathbb{R}^{2n}}\left\{\sum_{i=1}^n (\partial_{v_i}u)^TP \partial_{v_i}u\right\}f_{\infty}dxdv\\=4\sigma\int_{\mathbb{R}^{2n}}\left\{\sum_{i=1}^n\left(2 |\partial_{v_i}u_1|^2+2\nu  \partial_{v_i}u_1\cdot \partial_{v_i}u_2+2(\partial_{v_i}u_2)^T\frac{\partial^2  V}{\partial x^2}\partial_{v_i}u_2\right)\right\}f_{\infty}dxdv\\=
 8\sigma\int_{\mathbb{R}^{2n}}\left\{\sum_{i=1}^n |\partial_{v_i}u_1+\frac{\nu}{2} \partial_{v_i}u_2|^2\right\}f_{\infty}dxdv+8\sigma\int_{\mathbb{R}^{2n}}\left\{\sum_{i=1}^n (\partial_{v_i}u_2)^T\left(\frac{\partial^2  V}{\partial x^2}-\frac{\nu^2}{4}I\right)\partial_{v_i}u_2\right\}f_{\infty}dxdv\\
 \geq 8\sigma\int_{\mathbb{R}^{2n}}\left\{\sum_{i=1}^n (\partial_{v_i}u_2)^T\left(\frac{\partial^2  V}{\partial x^2}-\frac{\nu^2}{4}I\right)\partial_{v_i}u_2\right\}f_{\infty}dxdv.
 \end{multline}
 Now we consider the last term in \eqref{SKFP1}
 \begin{multline}\label{est:mix} \displaystyle 4\int_{\mathbb{R}^{2n}} u^T\begin{pmatrix}
0&0\\
0& \frac{\partial^2 (v\cdot \nabla_x V)}{\partial x^2}
\end{pmatrix} u f_{\infty}dxdv=4\int_{\mathbb{R}^{2n}} u_2^T\frac{\partial^2 (v\cdot \nabla_x V)}{\partial x^2}
 u_2 f_{\infty}dxdv\\=
  4\int_{\mathbb{R}^{2n}}\left\{\sum_{i,j=1}^n u_{2,i} v\cdot \nabla_x V_{ij}u_{2,j}\right\}f_{\infty}dxdv=4\int_{\mathbb{R}^{2n}}\left\{\sum_{i,j,k=1}^n u_{2,i} v_k V_{ijk}u_{2,j}\right\}f_{\infty}dxdv\\=-\frac{4\sigma}{\nu}\int_{\mathbb{R}^{2n}}\left\{\sum_{i,j,k=1}^n u_{2,i} V_{ijk}u_{2,j}(\partial_{v_k}f_{\infty})\right\}dxdv=\frac{4\sigma}{\nu}\int_{\mathbb{R}^{2n}}\left\{\sum_{i,j,k=1}^n \partial_{v_k}(u_{2,i} u_{2,j}) V_{ijk}\right\}f_{\infty}dxdv\\= \frac{4\sigma}{\nu}\int_{\mathbb{R}^{2n}}\left\{\sum_{i,j,k=1}^n (\partial_{v_k}u_{2,i}) u_{2,j} V_{ijk}+u_{2,i}(\partial_{v_k}u_{2,j})  V_{ijk}\right\}f_{\infty}dxdv\\=
 \frac{8\sigma}{\nu}\int_{\mathbb{R}^{2n}}\left\{\sum_{i,j,k=1}^n (\partial_{v_k}u_{2,i}) u_{2,j} V_{ijk}\right\}f_{\infty}dxdv=\frac{8\sigma}{\nu}\int_{\mathbb{R}^{2n}}\left\{\sum_{k=1}^n (\partial_{v_k} u_2)^T \frac{\partial^2(\partial_{x_k}V)}{\partial x^2} u_{2} \right\}f_{\infty}dxdv,
 \end{multline}
where we integrated by parts and used $\partial_{v_k} f_{\infty}=-\frac{\nu}{\sigma}v_k f_{\infty}$ and the notations $u_{2,i}\colonequals\partial_{v_i}\left(\frac{f}{f_{\infty}}\right), $ $V_{ij}\colonequals \partial_{x_ix_j}^2V,$ $V_{ijk}\colonequals \partial^3_{x_ix_jx_k}V.$  
 By \eqref{SKFP1}, \eqref{est:Tr}, \eqref{est:mix}, and \eqref{est:S} we obtain  
  \begin{multline*}\displaystyle\frac{d}{dt} S(f(t))+(\nu-\tau)S(f(t))\leq   -\tau S(f(t))\\-8\sigma\int_{\mathbb{R}^{2n}}\left\{\sum_{i=1}^n (\partial_{v_i}u_2)^T\left(\frac{\partial^2  V}{\partial x^2}-\frac{\nu^2}{4}I\right)\partial_{v_i}u_2\right\}f_{\infty}dxdv+
    \frac{8\sigma}{\nu}\int_{\mathbb{R}^{2n}}\left\{\sum_{i=1}^n (\partial_{v_i}u_{2})^T \frac{\partial^2(\partial_{x_i}V)}{\partial x^2} u_{2} \right\}f_{\infty}dxdv\\ \leq 
-4\tau\int_{\mathbb{R}^{2n}} u^T_{2}\left(\frac{\partial^2  V}{\partial x^2}-\frac{\nu^2}{4}I\right)u_{2}f_{\infty}dxdv -8\sigma\int_{\mathbb{R}^{2n}}\left\{\sum_{i=1}^n (\partial_{v_i}u_2)^T\left(\frac{\partial^2  V}{\partial x^2}-\frac{\nu^2}{4}I\right)\partial_{v_i}u_2\right\}f_{\infty}dxdv\\+\frac{8\sigma}{\nu}\int_{\mathbb{R}^{2n}}\left\{\sum_{i=1}^n (\partial_{v_i}u_2)^T \frac{\partial^2(\partial_{x_i}V)}{\partial x^2} u_{2} \right\}f_{\infty}dxdv
 \end{multline*}
 \begin{multline*}=-\frac{8\sigma}{\nu}\sum_{i=1}^n\int_{\mathbb{R}^{2n}} \left\{\nu (\partial_{v_i}u_2)^T\left(\frac{\partial^2  V}{\partial x^2}-\frac{\nu^2}{4}I\right)\partial_{v_i}u_2-(\partial_{v_i}u_2)^T \frac{\partial^2(\partial_{x_i}V)}{\partial x^2} u_{2}\right\}f_{\infty}dxdv\\-\frac{8\sigma}{\nu}\int_{\mathbb{R}^{2n}}\frac{\tau \nu}{2\sigma}u^T_{2}\left(\frac{\partial^2  V}{\partial x^2}-\frac{\nu^2}{4}I\right)u_{2}f_{\infty}dxdv.
 \end{multline*}
 The right hand side of this inequality  is a quadratic polynomial with respect to {$\partial_{v_i}u_{2}, \, \, i\in \{1,...,n\},$ and $u_2.$}  The corresponding matrix of this quadratic polynomial is
 \begin{equation}\label{matrix}  \begin{pmatrix}
 \nu \left(\frac{\partial^2  V}{\partial x^2}-\frac{\nu^2}{4}I\right)&0&...&0&-\frac{1}{2}\frac{\partial^2(\partial_{x_1}V)}{\partial x^2} \\
 0&\nu \left(\frac{\partial^2  V}{\partial x^2}-\frac{\nu^2}{4}I\right)&...&0&-\frac{1}{2}\frac{\partial^2(\partial_{x_2}V)}{\partial x^2}\\ 
 
 ...&...&...&...&...\\
 0&0&...&\nu \left(\frac{\partial^2  V}{\partial x^2}-\frac{\nu^2}{4}I\right)&-\frac{1}{2}\frac{\partial^2(\partial_{x_n}V)}{\partial x^2}\\
 -\frac{1}{2}\frac{\partial^2(\partial_{x_1}V)}{\partial x^2}&-\frac{1}{2}\frac{\partial^2(\partial_{x_2}V)}{\partial x^2}&...&-\frac{1}{2}\frac{\partial^2(\partial_{x_n}V)}{\partial x^2}&\frac{\tau \nu}{2\sigma}\left(\frac{\partial^2  V}{\partial x^2}-\frac{\nu^2}{4}I\right)
 \end{pmatrix}.
 \end{equation}
  The assumption $ \frac{\partial^2  V}{\partial x^2}-\frac{\nu^2}{4}I\geq \frac{\partial^2  V}{\partial x^2}+cI$ and the Assumption \ref{A1} imply that  \eqref{matrix} is positive semi-definite. \\

Thus we have obtained 
   $$\displaystyle\frac{d}{dt} S(f(t))+(\nu-\tau)S(f(t))\leq 0$$ 
   and  by Gr\"onwall's lemma 
   \begin{equation}\label{Gron.1}
   S(f(t))\leq e^{-(\nu-\tau)t}S(f_0).
   \end{equation}
   The estimate $P(x)\geq \eta I$  and the Poincar\'e inequality \eqref{Poincare} imply 
 \begin{equation}\label{L^2les S}
 \int_{\mathbb{R}^{2n}}\left(\frac{f(t)}{f_{\infty}}-1\right)^2 f_{\infty}dxdv\leq \frac{1}{2C_{PI}\eta} S(f(t))\leq \frac{1}{2C_{PI}\eta} e^{-(\nu-\tau)t}S(f_0).
 \end{equation} 
The matrix inequalities (see Lemma \ref{lem:mat.inq.} in  Appendix \ref{6.2})
\begin{equation}\label{mat.in1}
\frac{1}{1+\alpha_0+\sqrt{(1-\alpha_0)^2+\nu^2}}P \leq \begin{pmatrix}
I&0\\ 0& \frac{\partial^2  V}{\partial x^2}+(1-\alpha_0)I
\end{pmatrix}\leq \frac{1+\alpha_0+\sqrt{(1-\alpha_0)^2+\nu^2}}{4\alpha_0-\nu^2}P
\end{equation}   
 show that  $S(f(t))$ is equivalent to the functional $$\int_{\mathbb{R}^{2n}}\left|\nabla_{x}\left(\frac{f(t)}{f_{\infty}}\right)\right|^2 f_{\infty}dxdv+\int_{\mathbb{R}^{2n}}\nabla^T_{v}\left(\frac{f(t)}{f_{\infty}}\right)\left(\frac{\partial^2 V}{\partial x^2}+ (1-\alpha_0) I\right)\nabla_{v}\left(\frac{f(t)}{f_{\infty}}\right)f_{\infty} dxdv.$$ 
 This equivalence, and \eqref{L^2les S} let us 
 obtain \eqref{main.est}. 

  \subsubsection*{Case $(b)$:}
Assume $c=-\alpha_0=-\frac{\nu^2}{4}.$ 
Then by  Lemma \ref{mat.ine1} $(2c),$ for any $\varepsilon\in (0,\nu-\tau),$  the matrix $$P(x)\colonequals\begin{pmatrix}
2I&\nu I\\
\nu I& 2\frac{\partial^2 V(x)}{\partial x^2}+\varepsilon^2 I
\end{pmatrix}$$ satisfies {
\begin{equation}\label{ep QP+PQ^T}
Q(x)P(x)+P(x)Q^T(x)\geq (\nu-{\varepsilon})P(x) \, \, \, \, \text{and} \, \, \, P(x)\geq \eta I
\end{equation}
  for all $x \in \mathbb{R}^n$ and $\eta\colonequals {
1+\frac{\nu^2+2\varepsilon^2}{4}- \sqrt{(\frac{\nu^2+2\varepsilon^2}{4}-1)^2+\nu^2}} >0.$}
With this matrix we have  
\begin{multline}\label{S2:S}
S(f(t))=4\int_{\mathbb{R}^{2n}}|u_{1}+\frac{\nu}{2} u_{2}|^2f_{\infty}dxdv+4\int_{\mathbb{R}^{2n}} u^T_{2}\left(\frac{\partial^2  V}{\partial x^2}+\frac{2\varepsilon^2-\nu^2}{4}I\right)u_{2} f_{\infty}dxdv\\ \geq 4\int_{\mathbb{R}^{2n}} u^T_{2}\left(\frac{\partial^2  V}{\partial x^2}+\frac{2\varepsilon^2-\nu^2}{4}I\right)u_{2} f_{\infty}dxdv,
 \end{multline}
 \begin{multline}\label{S2:Tr}
  4\sigma\int_{\mathbb{R}^{2n}}\left\{\sum_{i=1}^n (\partial_{v_i}u)^T P \partial_{v_i}u\right\}f_{\infty}dxdv\\=
8\sigma\int_{\mathbb{R}^{2n}}\left\{\sum_{i=1}^n |\partial_{v_i}u_{1}+\frac{\nu}{2}\partial_{v_i} u_{2}|^2\right\}f_{\infty}dxdv+8\sigma\int_{\mathbb{R}^{2n}}\left\{\sum_{i=1}^n (\partial_{v_i}u_2)^T\left(\frac{\partial^2  V}{\partial x^2}+\frac{2\varepsilon^2-\nu^2}{4}I\right)\partial_{v_i}u_2\right\}f_{\infty}dxdv,
 \end{multline}
and {by using \eqref{[]P}, $\partial_{v_i}f_{\infty}=-\frac{\nu}{\sigma}v_i f_{\infty}:$}
\begin{multline}\label{S2:mix} -2\int_{\mathbb{R}^{2n}}  u^T\left\{[\nabla_x V\cdot \nabla_v-v\cdot \nabla_x+\nu v\cdot\nabla_v-\sigma \Delta_v]P\right\} u f_{\infty}dxdv\\=\frac{8\sigma}{\nu}\int_{\mathbb{R}^{2n}}\left\{\sum_{i=1}^n(\partial_{v_i}u_2)^T \frac{\partial^2(\partial_{x_i}V)}{\partial x^2} u_{2} \right\}f_{\infty}dxdv.
 \end{multline}
 \eqref{SKFP}, \eqref{ep QP+PQ^T}, \eqref{S2:S}, \eqref{S2:Tr}, \eqref{S2:mix}, and similar estimates as for Case $a)$ show that 
  \begin{multline*}
  \displaystyle\frac{d}{dt} S(f(t))+(\nu-\tau-\varepsilon)S(f(t))\\
  \leq  -\frac{8\sigma}{\nu}\sum_{i=1}^n\int_{\mathbb{R}^{2n}}\left\{ \nu (\partial_{v_i}u_2)^T\left(\frac{\partial^2  V}{\partial x^2}+\frac{2\varepsilon^2-\nu^2}{4}I\right)\partial_{v_i}u_{2}-(\partial_{v_i}u_2)^T \frac{\partial^2(\partial_{x_i}V)}{\partial x^2} u_{2}\right\}f_{\infty}dxdv\\-\frac{8\sigma}{\nu}\int_{\mathbb{R}^{2n}}\frac{\tau \nu}{2\sigma}u^T_{2}\left(\frac{\partial^2  V}{\partial x^2}+\frac{2\varepsilon^2-\nu^2}{4}I\right)u_{2}f_{\infty}dxdv.
  \end{multline*}
  The right hand side of this inequality  is a quadratic polynomial with respect to $\partial_{v_i}u_{2}, \, \, i\in \{1,...,n\},$ and  $u_2.$ The corresponding matrix of this quadratic polynomial is 
  \begin{equation}\label{S2:matrix}  \begin{pmatrix}
 \nu \left(\frac{\partial^2  V}{\partial x^2}+\frac{2\varepsilon^2-\nu^2}{4}I\right)&0&...&0&-\frac{1}{2}\frac{\partial^2(\partial_{x_1}V)}{\partial x^2} \\
 0&\nu \left(\frac{\partial^2  V}{\partial x^2}+\frac{2\varepsilon^2-\nu^2}{4}I\right)&...&0&-\frac{1}{2}\frac{\partial^2(\partial_{x_2}V)}{\partial x^2}\\ 

 ...&...&...&...&...\\
 0&0&...&\nu\left(\frac{\partial^2  V}{\partial x^2}+\frac{2\varepsilon^2-\nu^2}{4}I\right)&-\frac{1}{2}\frac{\partial^2(\partial_{x_n}V)}{\partial x^2}\\
 -\frac{1}{2}\frac{\partial^2(\partial_{x_1}V)}{\partial x^2}&-\frac{1}{2}\frac{\partial^2(\partial_{x_2}V)}{\partial x^2}&...&-\frac{1}{2}\frac{\partial^2(\partial_{x_n}V)}{\partial x^2}&\frac{\tau \nu}{2\sigma}\left(\frac{\partial^2  V}{\partial x^2}+\frac{2\varepsilon^2-\nu^2}{4}I\right)
 \end{pmatrix}.
 \end{equation}
    Because of $ \frac{\partial^2  V}{\partial x^2}+\frac{2\varepsilon^2-\nu^2}{4}I> \frac{\partial^2  V}{\partial x^2}+cI$ and  Assumption \ref{A1}, \eqref{S2:matrix} is positive definite and we get 
    $$\displaystyle\frac{d}{dt} S(f(t))+(\nu-\tau-\varepsilon)S(f(t))\leq 0$$
  { and by Gr\"onwall's lemma 
    \begin{equation}\label{Gron.2}
    S(f(t))\leq e^{-(\nu-\tau-\varepsilon)t}S(f_0).
    \end{equation}
     Similar to \eqref{L^2les S}, we have 
    \begin{equation}\label{L^2<eps.S}
  \int_{\mathbb{R}^{2n}}\left(\frac{f(t)}{f_{\infty}}-1\right)^2 f_{\infty}dxdv\leq \frac{1}{2C_{PI}\eta} S(f(t))\leq \frac{1}{2C_{PI}\eta} e^{-(\nu-\tau-\varepsilon)t}S(f_0).
    \end{equation}
  The functional 
   $$\int_{\mathbb{R}^{2n}}\left|\nabla_{x}\left(\frac{f(t)}{f_{\infty}}\right)\right|^2 f_{\infty}dxdv+\int_{\mathbb{R}^{2n}}\nabla^T_{v}\left(\frac{f(t)}{f_{\infty}}\right)\left(\frac{\partial^2 V}{\partial x^2}+ (1-\alpha_0) I\right)\nabla_{v}\left(\frac{f(t)}{f_{\infty}}\right)f_{\infty} dxdv$$  and $S(f(t))$ are equivalent because of (see Lemma \ref{lem:mat.inq.} in Appendix \ref{6.2})
\begin{equation}\label{mat.in2}
\frac{1}{1+\frac{\nu^2+2\varepsilon^2}{4}+\sqrt{\left(1-\frac{\nu^2+2\varepsilon^2}{4}\right)^2+\nu^2}}P \leq \begin{pmatrix}
I&0\\ 0& \frac{\partial^2  V}{\partial x^2}+(1-\alpha_0)I
\end{pmatrix}\leq \frac{1+\frac{\nu^2+2\varepsilon^2}{4}+\sqrt{\left(1-\frac{\nu^2+2\varepsilon^2}{4}\right)^2+\nu^2}}{2\varepsilon^2}P.
\end{equation}   
This equivalence, and \eqref{L^2<eps.S} imply \eqref{main.est}.}

 \subsubsection*{Case $(c)$ and $(d),$ exponential decay:}
Assume $c>-\frac{\nu^2}{4}.$ { For some $\gamma \geq 0$ to be chosen later,  we consider the  functional
 \begin{align}\label{Phi}
\Phi(f(t)) \colonequals & \gamma \int_{\mathbb{R}^{2n}}\left(\frac{f}{f_{\infty}}-1\right)^2 f_{\infty}d xdv+S(f(t))\nonumber\\= &\gamma \int_{\mathbb{R}^{2n}}\left(\frac{f}{f_{\infty}}-1\right)^2 f_{\infty}dxdv+2\int_{\mathbb{R}^{2n}}u^TPuf_{\infty}dxdv.
 \end{align}}
 {Using \eqref{t.der.L^2} and \eqref{SKFP}  its time derivative reads}
 $$\frac{d \Phi(f(t))}{dt}=-4\sigma\int_{\mathbb{R}^{2n}}\left\{\sum_{i=1}^n (\partial_{v_i}u)^TP\partial_{v_i} u\right\}f_{\infty}dxdv-2\int_{\mathbb{R}^{2n}}  u^T\left\{QP+PQ^T+\gamma D\right\}u f_{\infty}dxdv$$
\begin{equation}\label{time.der}-2\int_{\mathbb{R}^{2n}}  u^T\left\{[\nabla_x V\cdot \nabla_v-v\cdot \nabla_x+\nu v\cdot\nabla_v-\sigma \Delta_v]P\right\} u f_{\infty}dxdv.
\end{equation}
 Let  $a,$ to be chosen later, be any number such that $a\geq c+\frac{\nu^2}{4}>0$ and $a+\alpha_0>\frac{\nu^2}{4}.$ 
We consider the matrix 
\begin{equation}\label{P-bc}
P(x)\colonequals \begin{pmatrix}
2I&\nu I\\
\nu I& 2\frac{\partial^2 V(x)}{\partial x^2}+2aI
\end{pmatrix}.
\end{equation}
Then, by Lemma \ref{gen.mat.inq} we have
\begin{equation}\label{aux,mat.ineq1}
Q(x)P(x)+P(x)Q^T(x)+\gamma D\geq (\nu-\delta)P(x), \, \, \, \, \, \forall x \in \mathbb{R}^d,
\end{equation}
with a constant $\delta$ defined in \eqref{delta}. If $\gamma$  is large enough, \eqref{delta} shows that  $\delta\in (0,\nu-\tau).$  

The choice of the matrix $P$ in \eqref{P-bc}, \eqref{time.der}, and \eqref{aux,mat.ineq1} lets us estimate
\begin{multline}\label{timeder:Phi}
\frac{d \Phi(f(t))}{dt}\leq  -4\sigma\int_{\mathbb{R}^{2n}}\left\{\sum_{i=1}^n(\partial_{v_i}u)^T P \partial_{v_i}u \right\}f_{\infty}dxdv\\-(\nu-\delta)S(f(t))+4\int_{\mathbb{R}^{2n}} u^T\begin{pmatrix}
0&0\\
0& \frac{\partial^2 (v\cdot \nabla_x V)}{\partial x^2}
\end{pmatrix} u f_{\infty}dxdv.
\end{multline}
{Similar computations as for Case $(a)$ as well as \eqref{S2:S} (but with $\varepsilon^2=2a$) lead to} 
 \begin{multline*}\displaystyle\frac{d}{dt} \Phi(f(t))+(\nu-\delta-\tau)S(f(t))\\ \leq  -\frac{8\sigma}{\nu}\sum_{i=1}^n\int_{\mathbb{R}^{2n}}\left\{\nu(\partial_{v_i}u_2)^T\left(\frac{\partial^2  V}{\partial x^2}+\frac{4a-\nu^2}{4}I\right)\partial_{v_i}u_{2}-(\partial_{v_i}u_{2})^T \frac{\partial^2(\partial_{x_i}V)}{\partial x^2} u_{2}\right\}f_{\infty}dxdv\\-\frac{8\sigma}{\nu}\int_{\mathbb{R}^{2n}}\frac{\tau \nu}{2\sigma}u^T_{2}\left(\frac{\partial^2  V}{\partial x^2}+\frac{4a-\nu^2}{4}I\right)u_{2}f_{\infty}dxdv.
 \end{multline*}
{The two integrands of the right hand side are together a quadratic polynomial of $ \partial_{v_i}u_{2},$ $i\in \{1,...,n\},$ and $u_2,$} and  its corresponding matrix is 
 \begin{equation}\label{matrix3}  \begin{pmatrix}
 \nu \left(\frac{\partial^2  V}{\partial x^2}+\frac{4a-\nu^2}{4}I\right)&0&...&0&-\frac{1}{2}\frac{\partial^2(\partial_{x_1}V)}{\partial x^2} \\
 0&\nu\left(\frac{\partial^2  V}{\partial x^2}+\frac{4a-\nu^2}{4}I\right)&...&0&-\frac{1}{2}\frac{\partial^2(\partial_{x_2}V)}{\partial x^2}\\ 
 
 ...&...&...&...&...\\
 0&0&...&\nu \left(\frac{\partial^2  V}{\partial x^2}+\frac{4a-\nu^2}{4}I\right)&-\frac{1}{2}\frac{\partial^2(\partial_{x_n}V)}{\partial x^2}\\
 -\frac{1}{2}\frac{\partial^2(\partial_{x_1}V)}{\partial x^2}&-\frac{1}{2}\frac{\partial^2(\partial_{x_2}V)}{\partial x^2}&...&-\frac{1}{2}\frac{\partial^2(\partial_{x_n}V)}{\partial x^2}&\frac{\tau \nu}{2\sigma}\left(\frac{\partial^2  V}{\partial x^2}+\frac{4a-\nu^2}{4}I\right)
 \end{pmatrix}.
 \end{equation}
 Because of $a-\frac{\nu^2}{4}\geq c$ and Assumption \ref{A1}, the matrix \eqref{matrix3}  is positive semi-definite, thus, we have  \begin{equation}\label{phiS}
 \displaystyle\frac{d}{dt} \Phi(f(t))+(\nu-\tau-\delta)S(f(t))\leq 0.
 \end{equation} The estimate $P(x)\geq \eta I$  ($\eta>0 $ defined in \eqref{eta}) and the Poincar\'e inequality \eqref{Poincare} imply 
 \begin{equation*}
 \int_{\mathbb{R}^{2n}}\left(\frac{f}{f_{\infty}}-1\right)^2 f_{\infty}dxdv\leq \frac{1}{  2\eta C_{PI}} S(f(t))
  \end{equation*}
and so $$ \frac{1}{1+\frac{\gamma}{2\eta C_{PI}}}\Phi(f(t))\leq S(f(t)).$$  This estimate  and \eqref{phiS} let us  conclude 
 \begin{equation}\label{Gron.inq}
 \frac{d}{dt} \Phi(f(t))+2\lambda \Phi(f(t))\leq 0
 \end{equation}
  for  
  \begin{equation}\label{lambda max}2\lambda=\frac{\nu-\tau-\delta}{1+\frac{\gamma}{2\eta C_{PI}}}>0.
  \end{equation}
 By Gr\"onwall's lemma we obtain
  \begin{equation}\label{Gron.3}
  \Phi(f(t))\leq e^{-2\lambda t} \Phi(f_0).
  \end{equation}
One can check that (see Lemma \ref{lem:mat.inq.} in Appendix \ref{6.2})
\begin{equation}\label{mat.in3}
\frac{1}{a+\alpha_0+1+\sqrt{(a+\alpha_0-1)^2+\nu^2}} P \leq \begin{pmatrix}
I&0\\ 0& \frac{\partial^2  V}{\partial x^2}+(1-\alpha_0)I
\end{pmatrix}\leq \frac{a+\alpha_0+1+\sqrt{(a+\alpha_0-1)^2+\nu^2}}{4(a+\alpha_0)-\nu^2} P.
\end{equation} 
Hence, $S(f(t))$ is equivalent to the functional $$\int_{\mathbb{R}^{2n}}\left|\nabla_{x}\left(\frac{f(t)}{f_{\infty}}\right)\right|^2 f_{\infty}dxdv+\int_{\mathbb{R}^{2n}}\nabla^T_{v}\left(\frac{f(t)}{f_{\infty}}\right)\left(\frac{\partial^2 V}{\partial x^2}+ (1-\alpha_0) I\right)\nabla_{v}\left(\frac{f(t)}{f_{\infty}}\right)f_{\infty} dxdv.$$ 
 Subsequently, $\Phi(f(t))$ and  the functional on the left hand side of \eqref{main.est} are equivalent. This equivalence and \eqref{Gron.3} let us  obtain \eqref{main.est}.
  
   \subsubsection*{Case $(c)$ and $(d),$ estimated decay rate:}
Next, we shall estimate $\lambda$ from \eqref{lambda max} explicitly, and  we shall choose the parameters $a$ and $\gamma$ such that $\lambda$ is (rather) large.
  By \eqref{eta} and \eqref{quad.eq}, $\eta=\eta(a)$ and $\delta=\delta(a,\gamma)$ are  functions of $a\in [c+\frac{\nu^2}{4}, \infty)\bigcap (\frac{\nu^2}{4}-\alpha_0, \infty)$ and $\gamma\in [0,\infty).$   
  {Since $\delta>0,$ and $\eta$ is monotonically increasing up to 2, we have the following uniform estimate and  choice of the decay rate: }
  $$2\lambda\colonequals \sup_{a\in [c+\frac{\nu^2}{4}, \infty)\bigcap (\frac{\nu^2}{4}-\alpha_0, \infty),\, \gamma\geq 0}\frac{\nu-\tau-\delta(a,\gamma)}{1+\frac{\gamma}{2\eta(a) C_{PI}}}\leq \sup_{\gamma\geq 0}\frac{\nu-\tau}{1+\frac{\gamma}{4C_{PI}}}\leq \nu-\tau.$$ 
Next, we shall estimate this supremum (in fact it is a maximum). First we introduce a new variable  $s:=\displaystyle \frac{\gamma \sigma}{4a\sqrt{a+\alpha_0-\frac{\nu^2}{4}}}\in [0,\infty),$ then \begin{equation*}\delta(a,\gamma)= \frac{a}{\sqrt{a+\alpha_0-\frac{\nu^2}{4}}}(\sqrt{1+s^2}-s).
  \end{equation*}
  With the notations $\displaystyle A(a):=\frac{1+ a+\alpha_0+ \sqrt{( a+\alpha_0-1)^2+\nu^2}}{2\sigma C_{PI}}>0$ and $\displaystyle B(a):=\frac{a}{\sqrt{a+\alpha_0-\frac{\nu^2}{4}}}>0,$  we have $$2\lambda=\max_{a\in [c+\frac{\nu^2}{4}, \infty)\bigcap (\frac{\nu^2}{4}-\alpha_0, \infty),\, s\geq 0}\frac{\nu-\tau-B(a)(\sqrt{1+s^2}-s)}{1+A(a)B(a)s}.$$ 
 Next, we shall fix the parameter $a.$ To estimate $\lambda$ as accurately as possible,  we choose $a$ as the argmin of  $B(a)$ such that $\nu-\tau-B(a)(\sqrt{1+s^2}-s)$ is maximal with respect to $a.$     
 The minimal value of $B(a)$ is 
$$\displaystyle \min_{a\in [c+\frac{\nu^2}{4}, \infty)\bigcap (\frac{\nu^2}{4}-\alpha_0, \infty)} B(a)=\begin{cases}\displaystyle B(a_1)=\frac{c+\frac{\nu^2}{4}}{\sqrt{c+\alpha_0}} \, \,  \, \,  \, \, \, \, \text{ if } \,\, c+2\alpha_0>\frac{\nu^2}{4}\\ \displaystyle
B(a_2)=\sqrt{{\nu^2}-4\alpha_0} \, \, \text{ if } \,\, c+2\alpha_0\leq \frac{\nu^2}{4} \end{cases},$$
and this minimum is attained at $a_1\colonequals {c+\frac{\nu^2}{4}}$  if  $ c+2\alpha_0>\frac{\nu^2}{4}$ (i.e. in Case $(c)$), and $a_2\colonequals 
2(\frac{\nu^2}{4}-\alpha_0)$  if $ c+2\alpha_0\leq \frac{\nu^2}{4}$ (i.e. in Case $(d)$). \\

  If $c+2\alpha_0>\frac{\nu^2}{4},$ then  $c>-\alpha_0$ and so $a$ varies in $$ [c+\frac{\nu^2}{4}, \infty)\bigcap (\frac{\nu^2}{4}-\alpha_0, \infty)=[c+\frac{\nu^2}{4}, \infty)=[a_1, \infty).$$ Since
   $A(a)$ is increasing,   both  $A(a)$ and $B(a)$ attain their minimal values at $a_1.$ Thus, $a_1$ is optimal, i.e.   
$$  \max_{a}\frac{\nu-\tau-B(a)(\sqrt{1+s^2}-s)}{1+A(a)B(a)s}=\frac{\nu-\tau-B(a_1)(\sqrt{1+s^2}-s)}{1+A(a_1)B(a_1)s}.$$\\

If $c+2\alpha_0\leq \frac{\nu^2}{4},$  $a_2=2(\frac{\nu^2}{4}-\alpha_0)$ may not be optimal as $A(a)$ does not attain its minimum at this point, i.e.
$$  \max_{a}\frac{\nu-\tau-B(a)(\sqrt{1+s^2}-s)}{1+A(a)B(a)s}\geq\frac{\nu-\tau-B(a_2)(\sqrt{1+s^2}-s)}{1+A(a_2)B(a_1)s}.$$ But it is the optimal choice when $s=0$ and so it gives a good approximation   if $s$ is small.  From now on we assume that $a$ is fixed as
\begin{equation}\label{a}
\displaystyle a\colonequals \begin{cases} a_1 ={c+\frac{\nu^2}{4}} \, \, \, \, \, \, \,  \, \,  \, \, \, \, \text{ if } \,\, c+2\alpha_0>\frac{\nu^2}{4}\\ 
a_2=2(\frac{\nu^2}{4}-\alpha_0) \, \, \text{ if } \,\, c+2\alpha_0\leq \frac{\nu^2}{4} \end{cases}.
\end{equation}
Note that this choice is independent of $s.$\\
 Let $\displaystyle \Lambda(a,s)\colonequals\frac{\nu-\tau-B(a)(\sqrt{1+s^2}-s)}{1+A(a)B(a)s}$ and we seek its maximum with respect to $s\in [0, \infty).$ We compute
 \begin{multline}\label{part Lambda}
  \partial_s\Lambda(a,s)\\=\frac{B(a)}{(1+A(a)B(a)s)^2\sqrt{s^2+1}}\left([1-(\nu-\tau-B(a))A(a)]\sqrt{s^2+1}-A(a)B(a)(\sqrt{s^2+1}-1)-s \right).
  \end{multline}
 
 If $1-(\nu-\tau-B(a))A(a)\leq 0,$ then $\partial_s\Lambda(a,s)\leq 0$ which implies that $\Lambda(a,s)$ is a decreasing function of $s$ and the maximum in $[0,\infty) $ is attained at $s=0.$ \\
 
 If $1-(\nu-\tau-B(a))A(a)>0,$ then $\partial_s\Lambda(a,0)=B(a)[1-(\nu-\tau-B(a))A(a)]>0$ and  $\Lambda(a,s)$ is  increasing  in a neighborhood of $s=0.$ We also see $\partial_s\Lambda(a,s)$ is negative if $s$ is large enough (since $\nu-\tau>0$). This means that $\Lambda(a,s)$ starts to grow at $s=0$ and it decreases as $s\to \infty.$ Therefore, there is a point in $(0,\infty)$ at which $\Lambda(a,s)$ takes its maximum. Setting $\partial_s\Lambda(a,s)=0$ we obtain
 $$[1-(\nu-\tau)A(a)]\sqrt{s^2+1}-s+A(a)B(a)=0.$$
 It has only one solution in $(0,\infty)$ given by 
 \begin{small}
 \begin{equation}\label{s}
 s(a)=\begin{cases}
 \frac{A^2(a)B^2(a)-1}{2A(a)B(a)}& \text{ if }  (\nu-\tau)A(a)=2  \\
\frac{1}{\nu-\tau}\left[\left|\frac{(\nu-\tau)A(a)-1}{(\nu-\tau)A(a)-2}\right|\sqrt{B^2(a)+2(\nu-\tau)A^{-1}(a)-(\nu-\tau)^2}-\frac{B(a)}{(\nu-\tau)A(a)-2}\right] & \text{ if }  (\nu-\tau)A(a)\neq 2  
  \end{cases}
  \end{equation}
  \end{small}
and at this point  $\Lambda(a,s)$ attains its maximum with respect to $s.$

Considering the computations above, we conclude that the decay rate can be estimated by:
\begin{equation}
 \displaystyle 2\lambda=\begin{cases} \nu-\tau-B(a)& \text{ if } \nu-\tau\geq A^{-1}(a)+B(a)\\ 
\frac{\nu-\tau-B(a)(\sqrt{1+s^2(a)}-s(a))}{1+A(a)B(a)s(a)} & \text{ if }  \nu-\tau< A^{-1}(a)+B(a) \end{cases},
\end{equation}
where two cases correspond to the two cases discussed after \eqref{part Lambda}. Moreover, $a$ and $s(a)$ are defined in \eqref{a} and \eqref{s}, respectively. If we denote $A_1:=A(a_1),$ $A_2:=A(a_2),$ $s_1:=s(a_1)$ and $s_2:=s(a_2)$ and take into account that $B(a_1)=\frac{c+\frac{\nu^2}{4}}{\sqrt{c+\alpha_0}}$ and $B(a_2)=\sqrt{{\nu^2}-4\alpha_0}, $ we obtain the explicit decay rates stated in the theorem.
  \subsubsection*{Case (e):}
 Let $V(x)$ be a quadratic function of $x$ and  $\frac{\partial^2  V}{\partial x^2}$ be positive definite. Then, $\frac{\partial^2  (\partial_{x_i}V)}{\partial x^2}$ are zero matrices for all $i\in\{1,...,n\}.$  Thus, $V$ satisfies Assumption \ref{A1} with $\tau=0,$ $-c=\alpha_0>0.$  

 If $\alpha_0<\frac{\nu^2}{4},$ then  $c+2\alpha_0=\alpha_0<\frac{\nu^2}{4}$ which falls into  Case $(d).$ 
The constant  in the Poincar\'e inequality \eqref{Poincare} equals  $C_{PI}=\frac{\nu}{\sigma}\min\{1, \alpha_0\}$ (see \cite{OSI}). It lets us compute $A_2^{-1}$ explicitly: $$\displaystyle A_2^{-1}=\frac{2\nu \min\{1, \alpha_0\}}{1+ \frac{\nu^2}{2}-\alpha_0+ \sqrt{( \frac{\nu^2}{2}-\alpha_0-1)^2+\nu^2}}. $$
In Appendix \ref{6.3} we prove the following inequality:
   \begin{equation}\label{aA}
 \nu\geq A^{-1}_2+\sqrt{{\nu^2}-4\alpha_0}.\end{equation}  
 Thus Case $(d)$ implies \begin{equation}\label{rate a<}\lambda=  \frac{\nu-\sqrt{{\nu^2}-4\alpha_0}}{2}.
 \end{equation}

 If $\alpha_0\geq\frac{\nu^2}{4},$ the decay rate is explicit by Case $(a)$ and Case $(b):$  \begin{equation}\label{rate}\lambda=\begin{cases} {\frac{\nu}{2}}& \text{ if } \alpha_0>\frac{\nu^2}{4}\\
\frac{ \nu-\varepsilon}{2}& \text{ if } \alpha_0=\frac{\nu^2}{4}, \text{  for any } \varepsilon\in (0,\nu)
  \end{cases}.
 \end{equation}

 We now prove that the decay rates in \eqref{rate a<} and \eqref{rate} are sharp:  From Corollary \ref{Cor.}   $$\int_{\mathbb{R}^{2n}}\left(\frac{f(t)}{f_{\infty}}-1\right)^2 f_{\infty}dxdv 
   \leq C e^{-2\lambda t} \int_{\mathbb{R}^{2n}} \left(\frac{f_0}{f_{\infty}} -1\right)^2 \left(\left|\left|\frac{\partial^2V}{\partial x^2}\right|\right|^2+1\right) f_{\infty}dxdv,\, \, \, \, \, \, \forall t\geq t_0$$
   holds with the same $\lambda$ given in \eqref{rate a<} and \eqref{rate}.
 Since $\left|\left|\frac{\partial^2V}{\partial x^2}\right|\right|+1$ is  constant, this estimate implies 
 \begin{equation}\label{prop.l}
 \sup_{1\neq \frac{f_0}{f_{\infty}} \in L^2(\mathbb{R}^d, f_{\infty})}\frac{||f(t)/f_{\infty}-1||_{L^2(\mathbb{R}^d, f_{\infty})}}{||f_0/f_{\infty}-1||_{L^2(\mathbb{R}^d, f_{\infty})}}\leq \tilde{C} e^{-\lambda t},\, \, \, \, \, \, \forall t\geq t_0
 \end{equation}
 for some constant $\tilde{C}>0.$ On the one hand
this means that  the estimated decay rate $\lambda$ can not be larger than the (true) decay rate of the propagator norm given on the left hand side of \eqref{prop.l}. On the other hand, Proposition \ref{prop.}  gives the sharp decay rates for this propagator norm. The decay rates in \eqref{rate a<} and \eqref{rate} coincide with the  ones in Proposition \ref{prop.}  except in the case of $\alpha_0= \frac{\nu^2}{4}.$ Thus, the exponential decay rates  in Case $(a)$  and Case $(d)$ are  sharp.
When $\alpha_0= \frac{\nu^2}{4},$  Proposition \ref{prop.} provides the sharp decay $(1+t)e^{-\frac{\nu}{2} t}$ for the propagator norm.  
  Hence, \eqref{main.est} can hold with rates $\lambda=\frac{\nu-\varepsilon}{2}$ for any small fixed $\varepsilon \in (0,\nu),$ but it does not hold for $\varepsilon=0.$
  \end{proof}
 \subsection{Proof of Proposition \ref{prop.}}
\begin{proof}[{Proof of Proposition \ref{prop.}}] Let $V$ be a quadratic polynomial and   $\frac{\partial^2 V}{\partial x^2}\equalscolon M^{-1} \in \mathbb{R}^{n\times n} $ be positive definite.  Then there are   $x_0\in \mathbb{R}^n$ and $\mathcal{C}\in \mathbb{R}$ such that   $V(x)=\frac{{(x-x_0)}^T M^{-1}(x-x_0)}{2}+\mathcal{C},$ $\forall x \in \mathbb{R}^n.$  Since  the change $x \to x+x_0$ does not affect the supremum in \eqref{l^2 t^2} and only  the gradient of $V$ appears in \eqref{KFP}, without loss of generality  we assume  that  $x_0=0$ and $\mathcal{C}=0.$     
\subsubsection*{Step 1, reformulation as an ODE-problem:}
To this end we use Theorem \ref{prop.norm}. We check the conditions of this theorem for the kinetic Fokker-Planck equation. With the notation 
 $\xi=\begin{pmatrix}
 x\\v
 \end{pmatrix},$ we write
\begin{equation}\label{K}
E(\xi)=\frac{\nu}{\sigma}\left(V(x)+\frac{|v|^2}{2}\right)=\frac{\nu}{\sigma}\left(\frac{{x}^T M^{-1} x}{2}+\frac{|v|^2}{2}\right)=\frac{1}{2}\xi^T\begin{pmatrix}
  \frac{\nu}{\sigma}M^{-1}& 0\\
  0&\frac{\nu}{\sigma}I
  \end{pmatrix} \xi=\frac{\xi^TK^{-1}\xi}{2} 
  \end{equation}
  with $K^{-1}:=\frac{\nu}{\sigma}\begin{pmatrix}
 M^{-1}& 0\\
  0&I
  \end{pmatrix}.$
  From \eqref{D and R} we see that $ \text{Ker}D=\{(\psi,0)^T: \, \, \psi\in \mathbb{R}^n\}.$  Let $(\psi,0)^T\in  \text{Ker}D,$ then its image under $K^{-1}(D-R)$ is
  $$K^{-1}(D-R)\begin{pmatrix}
  \psi \\ 0
  \end{pmatrix}=\begin{pmatrix}
  0 & M^{-1}\\
  -I & \nu I
  \end{pmatrix}
   \begin{pmatrix}
  \psi \\ 0
  \end{pmatrix}=\begin{pmatrix}
  0\\-\psi
  \end{pmatrix} $$ and it is in $\text{Ker}D$ iff $\psi=0.$ Therefore, there is no non-trivial $K^{-1}(D-R)-$invariant subspace of $\text{Ker}D.$ 
 Next we compute the eigenvalues $\beta$ of $K^{-1/2}(D+R)K^{-1/2}=\begin{pmatrix}
0&-M^{-1/2}\\
M^{-1/2}&\nu I  
  \end{pmatrix}:$ 
\begin{multline*}
\begin{vmatrix}
-\beta I&-M^{-1/2}\\
M^{-1/2}&(\nu-\beta)I
\end{vmatrix}=\begin{vmatrix}
-\beta I&0\\
M^{-1/2}&(\nu-\beta)I-\beta^{-1}M^{-1}
\end{vmatrix}\\= \text{det}(\beta(\beta-\nu)I+M^{-1})=\prod_{i=1}^n(\beta^2-\nu \beta+\alpha_i)=0,
\end{multline*}
where $\alpha_i,$  
 $ i\in \{1,...,n\}$ denote the eigenvalues of $M^{-1}.$ By solving the latter equation, we find that the eigenvalues of $K^{-1/2}(D+R)K^{-1/2}$ are $\beta^{-}_{i}=\frac{\nu-\sqrt{\nu^2-4\alpha_i}}{2},$ $\beta^{+}_{i}=\frac{\nu+\sqrt{\nu^2-4\alpha_i}}{2},$  $ i\in \{1,...,n\}.$
  If $\alpha_0>0$ is the smallest eigenvalue of $M^{-1},$ then  $$\mu\colonequals \min_i\{\textsf{Re}(\beta_i): \beta_i \text{ is an eigenvalue of }  K^{-1/2}(D+R)K^{-1/2}\}=\begin{cases} \frac{\nu}{2}&  \text{ if } \alpha_0\geq \frac{\nu^2}{4}\\
 \frac{\nu-\sqrt{\nu^2-4\alpha_0}}{2} & \text{ if } \alpha_0<\frac{\nu^2}{4}  \end{cases}.$$ 
 Hence $\mu$ is positive, so $K^{-1/2}(D+R)K^{-1/2}$ and $(D+R)K^{-1}$ are positive stable.  Therefore,  Theorem \ref{prop.norm} applies to  the kinetic Fokker-Planck equation.
 \subsubsection*{Step 2, decay rates of the ODE-solution:}
We consider the ODE $$\dot{\xi}(t)=-K^{-1/2}(D+R)K^{-1/2}\xi$$ with the initial data $\xi(0)=\xi_0.$ Since $K^{-1/2}(D+R)K^{-1/2}$ is positive stable, the solution ${\xi}(t)$ is stable. To quantify the decay rate, we continue to analyze the eigenvalues of $K^{-1/2}(D+R)K^{-1/2}.$  
Let $m_i$ be the multiplicity of $\alpha_i>0$ as an eigenvalue of $M^{-1}$ (now  the $\alpha_i$ with $i\in\{1,...,\tilde{n}\}$ are labeled without multiplicity). Since $M^{-1}$ is symmetric, there are linearly independent eigenvectors $\psi_{ij}\in \mathbb{R}^n, \, \, j \in \{1,..., m_i\}$ of $M^{-1}$ corresponding to $\alpha_i.$ Then we can check that  the vectors \begin{equation}\label{eigvec1}
\begin{pmatrix}
 -\frac{\alpha_i^{1/2}}{\beta^{-}_{i}}\psi_{ij}\\
 \psi_{ij}
 \end{pmatrix}\in  \mathbb{R}^{2n}, \, \, j \in \{1,..., m_i\}
 \end{equation}
  are linearly independent eigenvectors of $K^{-1/2}(D+R)K^{-1/2}$ corresponding to $\beta^{-}_{i},$ $i\in\{1,...,\tilde{n}\}.$ Moreover, these vectors form a basis of the space of eigenvectors corresponding to $\beta^{-}_{i}.$ Similarly, the  vectors \begin{equation}\label{eigvec2}
\begin{pmatrix}
 -\frac{\alpha_i^{1/2}}{\beta^{+}_{i}}\psi_{ij}\\
 \psi_{ij}
 \end{pmatrix}\in  \mathbb{R}^{2n}, \, \, j \in \{1,..., m_i\}.
 \end{equation}
 satisfy the same property  for $\beta^{+}_{i}.$

 If $\alpha_i\neq\frac{\nu^2}{4}$ for all $i\in\{1,...,\tilde{n}\}$ (i.e., $\beta^{-}_{i}\neq \beta^{+}_{i}),$ the  algebraic multiplicities of $\beta^{-}_i$  and $\beta^{+}_i$ are  equal to $m_i.$ Then  $\beta^{-}_i$ (resp. $\beta^{+}_i$) has  $m_i$ eigenvectors given by \eqref{eigvec1} (resp. \eqref{eigvec2}). Thus, the geometric multiplicities of $\beta^{-}_i$ and $\beta^{+}_i$ also equal $m_i.$  In particular, $K^{-1/2}(D+R)K^{-1/2}$ is  diagonalizable.

  If $\alpha_{i_0}=\frac{\nu^2}{4}$ for some $i_0\in\{1,...,\tilde{n}\},$ then the algebraic multiplicity of $\beta^{-}_{i_0}=\beta^{+}_{i_0}=\frac{\nu}{2}$  equals $2m_{i_0}.$  Since the vectors \eqref{eigvec1} and \eqref{eigvec2} coincide in this case, the geometric multiplicity of $\frac{\nu}{2}$ equals $m_{i_0}.$ Thus, in this case, $\frac{\nu}{2}$ is a defective\footnote{An eigenvalue is \textit{defective} if its geometric multiplicity is strictly less than its algebraic multiplicity.} eigenvalue of $K^{-1/2}(D+R)K^{-1/2}$ with the corresponding eigenvectors
  \begin{equation}\label{eigen v 0}\begin{pmatrix}
 -\psi_{i_0 j}\\
 \psi_{i_0 j}
 \end{pmatrix}\in  \mathbb{R}^{2n}, \, \, j \in \{1,..., m_{i_0}\}.
 \end{equation} 
  By solving the following linear system (with respect to $\xi$)
 $$K^{-1/2}(D+R)K^{-1/2}\xi-\frac{\nu}{2}\xi=\begin{pmatrix}
-\frac{\nu}{2} I&-M^{-1/2}\\
M^{-1/2}&\frac{\nu}{2}I
\end{pmatrix}\xi=\begin{pmatrix}
 -\psi_{i_0 j}\\
 \psi_{i_0 j}
 \end{pmatrix}, \, \, \, \, \xi \in \mathbb{R}^{2d}, $$
 we find that  the solution
$
\xi= \begin{pmatrix}
 0\\
 \frac{2}{\nu}\psi_{i_0 j}
 \end{pmatrix}
 $
  is a generalized eigenvector  of $\frac{\nu}{2}$ corresponding to the eigenvector $\begin{pmatrix}
 -\psi_{i_0 j}\\
 \psi_{i_0 j}
 \end{pmatrix}.$
 Since $\psi_{i_0 j},$  $j \in \{1,..., m_{i_0}\}$ are linearly independent,
the vectors  \begin{equation}\label{geigv 0}
 \begin{pmatrix}
 0\\
 \frac{2}{\nu}\psi_{i_0 j}
 \end{pmatrix}, \, \, j \in \{1,..., m_{i_0}\}
 \end{equation} form a set of linearly independent generalized eigenvectors  of $\frac{\nu}{2}.$ Since the vectors in \eqref{eigen v 0} and \eqref{geigv 0} are linearly independent and their total number equals $2m_{i_0}$ (which is the algebraic multiplicity of $\frac{\nu}{2}$), we conclude that each eigenvector of $\frac{\nu}{2}$ has only one generalized eigenvector. Therefore, all Jordan blocks associated to $\frac{\nu}{2} $  have the same size $2\times 2.$ 
 In particular, if $\alpha_0=\frac{\nu^2}{4},$ then the eigenvalue $\mu=\frac{\nu}{2} $ is defective and the maximal size of the Jordan blocks associated to $\frac{\nu}{2} $ is 2.

  Then, the classical stability theory for ODEs shows that 
$$\sup_{1\neq \frac{f_0}{f_{\infty}} \in L^2(\mathbb{R}^d, f_{\infty})}\frac{||f(t)/f_{\infty}-1||_{L^2(\mathbb{R}^d, f_{\infty})}}{||f_0/f_{\infty}-1||_{L^2(\mathbb{R}^d, f_{\infty})}}=\sup_{0\neq \xi_0\in \mathbb{R}^d}\frac{||\xi(t)||_2}{||\xi_0||_2}\asymp \begin{cases}
e^{-\frac{\nu}{2}t},& \text{ if } \alpha_0>\frac{\nu^2}{4} \\
(1+ t)e^{-\frac{\nu}{2}t},& \text{ if } \alpha_0=\frac{\nu^2}{4}\\
 e^{-{\frac{\nu-\sqrt{\nu^2-4\alpha_0}}{2}t}},& \text{ if } \alpha_0<\frac{\nu^2}{4} 
\end{cases}$$  
 as $ t\to \infty.$

\end{proof}

\begin{Remark}
With the eigenvalues of $C\colonequals (D+R)K^{-1}$ $($see \eqref{GFP}, \eqref{K}$)$ obtained at the end of Step 1 in the above proof, the sharpness of the decay rate $\mu$ in the cases 1 and  3 of \eqref{l^2 t^2} would also follow from  \cite[Theorem 6.1]{AE}.
\end{Remark}
\bigskip


\subsection{Proof of Theorem \ref{Hyp.ellip.} and Corollary \ref{Cor.}}

\text{   }

\begin{proof}[{Proof of Theorem \ref{Hyp.ellip.}}]
\subsubsection*{Step 1, an auxiliary inequality:}
As we assume the matrix \eqref{Condition1} is positive semi-definite, then the following submatrices of \eqref{Condition1} are positive semi-definite:
 \begin{equation*}
Y_k\colonequals \begin{pmatrix}
 \nu \left(\frac{\partial^2  V}{\partial x^2}+cI\right)&-\frac{1}{2}\frac{\partial^2(\partial_{x_k}V)}{\partial x^2} \\
 -\frac{1}{2}\frac{\partial^2(\partial_{x_k}V)}{\partial x^2}&\frac{\tau \nu}{2\sigma}\left(\frac{\partial^2  V}{\partial x^2}+cI\right)
 \end{pmatrix} \in \mathbb{R}^{2n\times 2n}, \, \, \, \, \, \, k\in \{1,...,n\}.
 \end{equation*}
Letting  $\delta>0,$ we consider
\begin{equation*}
X_{\delta}\colonequals \begin{pmatrix}I&\delta I\\\delta I&\delta^2 I
\end{pmatrix} 
 \otimes
 \left(\frac{\partial^2  V}{\partial x^2}+cI\right)=\begin{pmatrix}
 \frac{\partial^2  V}{\partial x^2}+cI&\delta \frac{\partial^2  V}{\partial x^2}+\delta cI \\ \delta \frac{\partial^2  V}{\partial x^2}+\delta cI&\delta^2 \frac{\partial^2  V}{\partial x^2}+\delta^2 cI
 \end{pmatrix}\in \mathbb{R}^{2n\times 2n}.
 \end{equation*}
$X_{\delta}$ is positive semi-definite  as it is the Kronecker product \cite[Corollary 4.2.13]{MA} of two positive semi-definite matrices. Hence, we have for all $k \in \{1,...,n\}:$  
\begin{equation*}
\mathrm{Tr}(X^{1/2}_{\delta}Y_kX^{1/2}_{\delta})=\mathrm{Tr}(X_{\delta}Y_k)=(\nu+\delta^2\frac{ \tau \nu}{2\sigma}) \mathrm{Tr} \left[\left(\frac{\partial^2  V}{\partial x^2}+c I\right)^2\right]-\delta \mathrm{Tr}\left[\left(\frac{\partial^2V}{\partial x^2}+cI\right)\frac{\partial^2(\partial_{x_k}V)}{\partial x^2}\right]\geq 0. 
\end{equation*}
This implies
\begin{equation}\label{nonopt.tr}
\frac{2\sigma\nu+\delta^2 \tau \nu}{2\sigma\delta} \mathrm{Tr} \left[\left(\frac{\partial^2  V}{\partial x^2}+c I\right)^2\right]\geq  \mathrm{Tr}\left[\left(\frac{\partial^2V}{\partial x^2}+c I\right)\frac{\partial^2(\partial_{x_k}V)}{\partial x^2}\right] 
\end{equation}
and by minimizing  the constant on the left hand side of \eqref{nonopt.tr} with respect to $\delta$ (i.e.,\,by choosing $\delta=\sqrt{\frac{2\sigma}{\tau}}$), we obtain  
\begin{equation}\label{opt.tr}
\sqrt{\frac{2\tau \nu^2}{\sigma}}\mathrm{Tr} \left[\left(\frac{\partial^2  V(x)}{\partial x^2}+c I\right)^2\right]\geq  \mathrm{Tr}\left[\left(\frac{\partial^2V(x)}{\partial x^2}+cI\right)\frac{\partial^2(\partial_{x_k}V(x))}{\partial x^2}\right] \, \, \, \, \, \, \, \text{for all} \,\, \, \, x \in \mathbb{R}^n.
\end{equation}
\subsubsection*{Step 2, growth estimate for the r.h.s. of \eqref{Vhyp.ell.1}, \eqref{Vhyp.ell.2}:} We denote $\displaystyle u_1\colonequals \nabla_x \left(\frac{f(t)}{f_{\infty}}\right),$  $\displaystyle u_2\colonequals \nabla_v \left(\frac{f(t)}{f_{\infty}}\right),$ and $u\colonequals \begin{pmatrix}
u_1\\u_2
\end{pmatrix}.$
Since $\displaystyle  \frac{f(t)}{f_{\infty}}-1 $ satisfies  $$\partial_t \left(\frac{f(t)}{f_{\infty}}-1\right)=-v \cdot \nabla_x \left(\frac{f(t)}{f_{\infty}}-1\right)+\nabla_x V \cdot \nabla_v \left(\frac{f(t)}{f_{\infty}}-1\right)+\sigma \Delta_v \left(\frac{f(t)}{f_{\infty}}-1\right)-\nu v\cdot\nabla_v \left(\frac{f(t)}{f_{\infty}}-1\right)$$
 and by integrating by parts, we obtain 
\begin{equation}\label{t.der.}
\frac{d}{dt} \displaystyle \int_{\mathbb{R}^{2n}} \left(\frac{f(t)}{f_{\infty}} -1\right)^2f_{\infty}dxdv=-2\sigma\displaystyle \int_{\mathbb{R}^{2n}} |u_2|^2 f_{\infty}dxdv.
\end{equation}
Next, we compute (with $||\cdot||$ denoting the Frobenius norm) 
\begin{multline}\label{t.der.V}
\frac{d}{dt} \displaystyle \int_{\mathbb{R}^{2n}} \left(\frac{f(t)}{f_{\infty}} -1\right)^2 \left|\left|\frac{\partial^2V}{\partial x^2}+c I\right|\right|^2 f_{\infty}dxdv\\
=2\int_{\mathbb{R}^{2n}} \left(\frac{f(t)}{f_{\infty}}-1 \right)\partial_t \left(\frac{f(t)}{f_{\infty}}-1\right) \left|\left|\frac{\partial^2V}{\partial x^2}+c I\right|\right|^2 f_{\infty}dxdv\\=2\int_{\mathbb{R}^{2n}} \left(\frac{f(t)}{f_{\infty}} -1\right)\left[-v \cdot \nabla_x \left(\frac{f(t)}{f_{\infty}}-1\right)+\nabla_x V \cdot \nabla_v \left(\frac{f(t)}{f_{\infty}}-1\right)\right] \left|\left|\frac{\partial^2V}{\partial x^2}+c I\right|\right|^2 f_{\infty}dxdv\\
+2\int_{\mathbb{R}^{2n}} \left(\frac{f(t)}{f_{\infty}} -1\right)\left[\sigma \Delta_v \left(\frac{f(t)}{f_{\infty}}-1\right)-\nu v\cdot\nabla_v \left(\frac{f(t)}{f_{\infty}}-1\right)\right]\left|\left|\frac{\partial^2V}{\partial x^2}+c I\right|\right|^2 f_{\infty}dxdv.
\end{multline}
 Integrating by parts with respect to $v,$ we obtain
\begin{multline}\label{L term}
2\int_{\mathbb{R}^{2n}} \left(\frac{f(t)}{f_{\infty}} -1\right)\left[\sigma \Delta_v \left(\frac{f(t)}{f_{\infty}}-1\right)-\nu v\cdot\nabla_v \left(\frac{f(t)}{f_{\infty}}-1\right)\right] \left|\left|\frac{\partial^2V}{\partial x^2}+c I\right|\right|^2 f_{\infty}dxdv\\=
-2\sigma\displaystyle \int_{\mathbb{R}^{2n}} |u_2|^2 \left|\left|\frac{\partial^2V}{\partial x^2}+c I\right|\right|^2 f_{\infty}dxdv.
\end{multline}
Next, we work on the term in the second line of \eqref{t.der.V}:
\begin{multline}\label{T term}
2\int_{\mathbb{R}^{2n}} \left(\frac{f(t)}{f_{\infty}} -1\right)\left[-v \cdot \nabla_x \left(\frac{f(t)}{f_{\infty}}-1\right)+\nabla_x V \cdot \nabla_v \left(\frac{f(t)}{f_{\infty}}-1\right)\right]\left|\left|\frac{\partial^2V}{\partial x^2}+c I\right|\right|^2 f_{\infty}dxdv\\
=\int_{\mathbb{R}^{2n}} \left(-v \cdot \nabla_x \left(\frac{f(t)}{f_{\infty}}-1\right)^2+\nabla_x V \cdot \nabla_v \left(\frac{f(t)}{f_{\infty}}-1\right)^2\right)\left|\left|\frac{\partial^2V}{\partial x^2}+c I\right|\right|^2 f_{\infty}dxdv\\
=\int_{\mathbb{R}^{2n}} \left(\frac{f(t)}{f_{\infty}}-1\right)^2 \left[v \cdot \nabla_x \left(\left|\left|\frac{\partial^2V}{\partial x^2}+c I\right|\right|^2 f_{\infty}\right)-\nabla_x V \cdot \nabla_v \left( \left|\left|\frac{\partial^2V}{\partial x^2}+c I\right|\right|^2 f_{\infty}\right)\right] dxdv\\
=\int_{\mathbb{R}^{2n}} \left(\frac{f(t)}{f_{\infty}}-1\right)^2 v \cdot \nabla_x \left( \left|\left|\frac{\partial^2V}{\partial x^2}+c I\right|\right|^2\right) f_{\infty}dxdv\\
=
\frac{2\sigma}{\nu}\int_{\mathbb{R}^{2n}} \left(\frac{f(t)}{f_{\infty}} -1\right) u_2\cdot \nabla_x\left(\left|\left|\frac{\partial^2V}{\partial x^2}+c I\right|\right|^2\right) f_{\infty}dxdv\\=\frac{2\sigma}{\nu}\int_{\mathbb{R}^{2n}} \left(\frac{f(t)}{f_{\infty}} -1\right)\sum_{k=1}^n u_{2,k} \partial_{x_k}\left(\left|\left|\frac{\partial^2V}{\partial x^2}+c I\right|\right|^2\right) f_{\infty}dxdv\\= \frac{4\sigma}{\nu}\int_{\mathbb{R}^{2n}} \left(\frac{f(t)}{f_{\infty}} -1\right)\left\{\sum_{k=1}^n u_{2,k}\sum_{i,j=1}^n(\partial^2_{x_i x_j}V+\delta_{ij}c)\partial^2_{x_i x_j}(\partial_{x_k}V)\right\} f_{\infty}dxdv,
\end{multline}
where we integrated by parts twice, and used $-\frac{\nu}{\sigma}vf_{\infty}=\nabla_vf_{\infty}$ and the notations $$u_{2,k}\colonequals \partial_{v_k} \left(\frac{f(t)}{f_{\infty}} \right) \,\,\, \, \text{ and } \, \,\, \, \delta_{ij}\colonequals \begin{cases} 1 \, \, \, \, \text{ if }\,\, i=j\\
0 \, \, \, \, \text{ if }\,\, i\neq j\end{cases}.$$ 
Using the identity 
$$\displaystyle \sum_{i,j=1}^n(\partial^2_{x_i x_j}V+\delta_{ij}c)\partial^2_{x_i x_j}(\partial_{x_k}V)=\mathrm{Tr}\left[\left(\frac{\partial^2V}{\partial x^2}+cI\right)\frac{\partial^2(\partial_{x_k}V)}{\partial x^2}\right],$$ the estimate \eqref{opt.tr}, and the discrete H\"older inequality, \eqref{T term} can be estimated as
\begin{multline}\label{last,term}
 \frac{4\sigma}{\nu}\int_{\mathbb{R}^{2n}} \left(\frac{f(t)}{f_{\infty}} -1\right)\left\{\sum_{i,j,k=1}^n u_{2,k}(\partial^2_{x_i x_j}V+\delta_{ij}c)\partial^2_{x_i x_j}(\partial_{x_k}V)\right\} f_{\infty}dxdv\\=\frac{4\sigma}{\nu}\int_{\mathbb{R}^{2n}} \left(\frac{f(t)}{f_{\infty}} -1\right)\left\{\sum_{k=1}^n u_{2,k}\mathrm{Tr}\left[\left(\frac{\partial^2V}{\partial x^2}+c I\right)\frac{\partial^2(\partial_{x_k}V)}{\partial x^2}\right]\right\} f_{\infty}dxdv\\ \leq 4\sqrt{2 \sigma \tau}\int_{\mathbb{R}^{2n}} \left(\frac{f(t)}{f_{\infty}} -1\right)\left\{\sum_{k=1}^n |u_{2,k}|\mathrm{Tr}\left[\left(\frac{\partial^2V}{\partial x^2}+c I\right)^2\right]\right\} f_{\infty}dxdv \\ \leq 4\sqrt{2 \sigma \tau n}\int_{\mathbb{R}^{2n}} \left(\frac{f(t)}{f_{\infty}} -1\right)|u_{2}|\mathrm{Tr}\left[\left(\frac{\partial^2V}{\partial x^2}+c I\right)^2\right] f_{\infty}dxdv\\
\leq \sigma \int_{\mathbb{R}^{2n}} |u_{2}|^2\mathrm{Tr}\left[\left(\frac{\partial^2V}{\partial x^2}+c I\right)^2\right] f_{\infty}dxdv+8\tau n\int_{\mathbb{R}^{2n}} \left(\frac{f(t)}{f_{\infty}}-1 \right)^2 \mathrm{Tr}\left[\left(\frac{\partial^2V}{\partial x^2}+c I\right)^2\right] f_{\infty}dxdv.  
\end{multline}
Combining the equations from \eqref{t.der.V} to \eqref{last,term} and the identity 
\begin{equation*}\label{TR}
\displaystyle \left|\left|\frac{\partial^2V}{\partial x^2}+c I\right|\right|^2=\mathrm{Tr}\left[\left(\frac{\partial^2V}{\partial x^2}+c I\right)^2\right],
\end{equation*}
 we get 
\begin{multline}\label{t.d.f^2V}
\frac{d}{dt} \displaystyle \int_{\mathbb{R}^{2n}} \left(\frac{f(t)}{f_{\infty}} -1\right)^2 \left|\left|\frac{\partial^2V}{\partial x^2}+c I\right|\right|^2 f_{\infty}dxdv\\ \leq -\sigma \int_{\mathbb{R}^{2n}} |u_{2}|^2\left|\left|\frac{\partial^2V}{\partial x^2}+cI\right|\right|^2 f_{\infty}dxdv+8\tau n\int_{\mathbb{R}^{2n}} \left(\frac{f(t)}{f_{\infty}} -1\right)^2 \left|\left|\frac{\partial^2V}{\partial x^2}+c I\right|\right|^2 f_{\infty}dxdv.
\end{multline}
\eqref{t.d.f^2V} can be reformulated as 
\begin{multline}\label{t.d.f^2V1}
\frac{d}{dt}\left( e^{-8\tau n t}\displaystyle \int_{\mathbb{R}^{2n}} \left(\frac{f(t)}{f_{\infty}} -1\right)^2 \left|\left|\frac{\partial^2V}{\partial x^2}+c I\right|\right|^2 f_{\infty}dxdv \right)  \leq -\sigma e^{-8\tau n t} \int_{\mathbb{R}^{2n}} |u_{2}|^2\left|\left|\frac{\partial^2V}{\partial x^2}+c I\right|\right|^2 f_{\infty}dxdv. 
\end{multline}
\subsubsection*{Step 3, $t-$dependent functional $\Psi$:} In order to prove the short-time regularization of \eqref{Vhyp.ell.1} and \eqref{Vhyp.ell.2} we introduce now an auxiliary functional that depends explicitly on time. Our strategy  is the generalization of the approach in \cite[Theorem A.12]{Vil}, \cite[Theorem 1.1]{Herau}, \cite[Theorem 4.8]{AE}.

For $t \in (0,t_0],$ we consider the following functional
\begin{equation}\label{func.t} \displaystyle
\Psi(t,f(t))\colonequals \displaystyle  \int_{\mathbb{R}^{2n}} \left(\frac{f(t)}{f_{\infty}} -1\right)^2 \left(\gamma_1 e^{-8\tau n t}\left|\left|\frac{\partial^2V}{\partial x^2}+c I\right|\right|^2+\gamma_2\right) f_{\infty}dxdv+\int_{\mathbb{R}^{2n}} u^TP u f_{\infty}dxdv,
\end{equation} 
with the $t-$ and $x-$dependent matrix in $\mathbb{R}^{2n\times 2n},$
\begin{equation}\label{P(x,t)}
P=P(t,x)\colonequals\begin{pmatrix}
2\varepsilon^3 t^3I&\varepsilon^2 t^2I\\\varepsilon^2 t^2I& 2\varepsilon tI+t(\frac{\partial^2V}{\partial x^2}+cI)
\end{pmatrix}.
\end{equation}
 $\varepsilon,$
 $\gamma_1,$ and $\gamma_2$  are  positive constants which we shall fix  later.
 We note that, for all $t\in (0,t_0],$ 
\begin{equation}\label{P>0}
P(t,x)\geq \begin{pmatrix}
\varepsilon^3 t^3I&0\\0& t(\frac{\partial^2V}{\partial x^2}+cI)+\varepsilon tI
\end{pmatrix}>\begin{pmatrix}
\varepsilon^3 t^3I&0\\0& t(\frac{\partial^2V}{\partial x^2}+cI)
\end{pmatrix}\geq 0
\end{equation}
as  $\frac{\partial^2V}{\partial x^2}+cI$ is positive semi-definite. Thus,  $\Psi(t,f(t))$ is non-negative and satisfies \begin{multline}\label{func.t.pos.}
\displaystyle \Psi(t,  f(t))\geq \displaystyle \displaystyle  \int_{\mathbb{R}^{2n}} \left(\frac{f(t)}{f_{\infty}} -1\right)^2 \left(\gamma_1 e^{-8\tau n t} \left|\left|\frac{\partial^2V}{\partial x^2}+c I\right|\right|^2+\gamma_2\right) f_{\infty}dxdv+\varepsilon^3t^3\displaystyle \int_{\mathbb{R}^{2n}}|u_1|^2 f_{\infty} dx dv\\+ t\int_{\mathbb{R}^{2n}} u_2^T \left(\frac{\partial^2V}{\partial x^2}+(c+\varepsilon)I\right) u_2 f_{\infty} dxdv. 
\end{multline}
Our goal is to show that $\Psi(t,  f(t))$ decreases. To this end  we estimate the time derivative of the second term in \eqref{func.t}. 
 First, \eqref{derivS} yields 
\begin{multline}\label{derv.S.hypel.}
\frac{d}{dt}\int_{\mathbb{R}^{2n}}u^TPuf_{\infty}dxdv\\ =-2\sigma\int_{\mathbb{R}^{2n}}\left\{\sum_{i=1}^n (\partial_{v_i}u)^TP\partial_{v_i}u\right\}f_{\infty}dxdv-\int_{\mathbb{R}^{2n}}  u^T\left\{QP+PQ^T-\partial _t P\right\}u f_{\infty}dxdv\\ -\int_{\mathbb{R}^{2n}}  u^T\left\{[\nabla_x V\cdot \nabla_v-v\cdot \nabla_x+\nu v\cdot\nabla_v-\sigma \Delta_v]P\right\} u f_{\infty}dxdv,
\end{multline} with $Q= \begin{pmatrix}
0&I\\
-\frac{\partial^2 V(x)}{\partial x^2}&\nu I
\end{pmatrix}.$
We consider each terms of \eqref{derv.S.hypel.}. Because of \eqref{P>0}, the first term can be estimated as
\begin{equation}\label{second.deriv.}
-2\sigma\int_{\mathbb{R}^{2n}}\left\{\sum_{i=1}^n (\partial_{v_i}u)^TP\partial_{v_i}u\right\}f_{\infty}dxdv\leq -2t\sigma\int_{\mathbb{R}^{2n}}\left\{\sum_{i=1}^n (\partial_{v_i}u_2)^T\left(\frac{\partial^2V}{\partial x^2}+c I\right)
\partial_{v_i}u_2\right\}f_{\infty}dxdv.
\end{equation}
For the third term of \eqref{derv.S.hypel.} we have $$[\nabla_x V\cdot \nabla_v-v\cdot \nabla_x+\nu v\cdot\nabla_v-\sigma \Delta_v]P=\begin{pmatrix}
0&0\\
0& -t\frac{\partial^2 (v\cdot \nabla_x V)}{\partial x^2}
\end{pmatrix}$$
and using $vf_{\infty}=-\frac{\sigma}{\nu}\nabla_v f_{\infty}$ yields 
\begin{multline}\label{est.mix.V}
-\int_{\mathbb{R}^{2n}}  u^T\left\{[\nabla_x V\cdot \nabla_v-v\cdot \nabla_x+\nu v\cdot\nabla_v-\sigma \Delta_v]P\right\} u f_{\infty}dxdv\\=\frac{2t\sigma}{\nu}\int_{\mathbb{R}^{2n}}\left\{\sum_{i=1}^n (\partial_{v_i}u_{2})^T \frac{\partial^2(\partial_{x_i}V)}{\partial x^2} u_{2} \right\}f_{\infty}dxdv.
\end{multline} 
For the second term of \eqref{derv.S.hypel.} we compute 
\begin{multline}
-\int_{\mathbb{R}^{2n}} u^T\left\{QP+PQ^T-\partial_t P\right\}uf_{\infty}dxdv\\
=-\int_{\mathbb{R}^{2n}} \begin{pmatrix} u_1\\ u_2 \end{pmatrix}^T \small{\begin{pmatrix}
0&
(t-2\varepsilon^3t^3)\left(\frac{\partial^2V}{\partial x^2}+c I\right)\\
(t-2\varepsilon^3t^3)\left(\frac{\partial^2V}{\partial x^2}+c I\right)& (-1+2\nu t-2\varepsilon^2 t^2)\left(\frac{\partial^2V}{\partial x^2}+c I\right)
\end{pmatrix}}\begin{pmatrix} u_1\\ u_2 \end{pmatrix}f_{\infty}dxdv
\\ -\int_{\mathbb{R}^{2n}} \begin{pmatrix} u_1\\ u_2 \end{pmatrix}^T \small{\begin{pmatrix}
2\varepsilon^2t^2(1-3\varepsilon)I&
[2c\varepsilon^3t^3+\nu \varepsilon^2 t^2+2(1-\varepsilon)\varepsilon t]I\\
[2c \varepsilon^3t^3+\nu \varepsilon^2 t^2+2(1-\varepsilon)\varepsilon t]I&[2c \varepsilon^2t^2+4\varepsilon \nu t-2\varepsilon]I
\end{pmatrix}}\begin{pmatrix} u_1\\ u_2 \end{pmatrix}f_{\infty}dxdv.
\end{multline}
Using  the estimates
\begin{multline*}
-(t-2\varepsilon^3t^3)\int_{\mathbb{R}^{2n}}u_1^T\left(\frac{\partial^2V}{\partial x^2}+cI\right)u_2f_{\infty}dxdv\\ \leq  \varepsilon^3 t^2 |1-2\varepsilon^3t^2|\int_{\mathbb{R}^{2n}}|u_1|^2f_{\infty}dxdv+\frac{|1-2\varepsilon^3t^2|}{4\varepsilon^3}\int_{\mathbb{R}^{2n}}|u_2|^2 \left|\left|\frac{\partial^2V}{\partial x^2}+c I\right|\right|^2f_{\infty}dxdv
\end{multline*}
and 
\begin{multline*}
-(-1+2\nu t-2\varepsilon^2 t^2)\int_{\mathbb{R}^{2n}}u_2^T\left(\frac{\partial^2V}{\partial x^2}+c I\right)u_2f_{\infty}dxdv\\ \leq |1-2\nu t+2\varepsilon^2 t^2|  \int_{\mathbb{R}^{2n}}|u_2|^2 \left|\left|\frac{\partial^2V}{\partial x^2}+c I\right|\right|f_{\infty}dxdv,
\end{multline*}
we get 
\begin{multline}\label{QP+PQ^T.V}
-\int_{\mathbb{R}^{2n}} u^T\left\{QP+PQ^T-\partial_t P\right\}uf_{\infty}dxdv\\
\leq \int_{\mathbb{R}^{2n}}  |u_2|^2\left[\frac{|1-2\varepsilon^3t^2|}{2\varepsilon^3}\left|\left|\frac{\partial^2V}{\partial x^2}+c I\right|\right|^2+|1-2\nu t+2\varepsilon^2 t^2|\left|\left|\frac{\partial^2V}{\partial x^2}+c I\right|\right|\right]
f_{\infty}dxdv
\\-\int_{\mathbb{R}^{2n}} \begin{pmatrix} u_1\\ u_2 \end{pmatrix}^T \small{\begin{pmatrix}
2\varepsilon^2t^2(1-3\varepsilon-\varepsilon|1-2\varepsilon^2t^2|)I&
[2c\varepsilon^3t^3+\nu \varepsilon^2 t^2+2(1-\varepsilon)\varepsilon t]I\\
[2c \varepsilon^3t^3+\nu \varepsilon^2 t^2+2(1-\varepsilon)\varepsilon t]I&[2c \varepsilon^2t^2+4\varepsilon \nu t-2\varepsilon]I
\end{pmatrix}}\begin{pmatrix} u_1\\ u_2 \end{pmatrix}f_{\infty}dxdv.
\end{multline}
We fix $\varepsilon=\varepsilon(t_0)>0$ so that the element in the upper left corner of the matrix in  \eqref{QP+PQ^T.V} is positive for $t>0;$ more precisely we require
  \begin{equation}\label{eps(t_0)}
1-3\varepsilon-\varepsilon|1-2\varepsilon^2t^2|>0 \, \, \, \, \, \text{ for all }\, \, \,  t \in [0,t_0]. 
\end{equation}
  Then,  the matrix in  the last line of \eqref{QP+PQ^T.V} can be estimated as 
\begin{multline*}
\begin{pmatrix}
2\varepsilon^2t^2(1-3\varepsilon-\varepsilon|1-2\varepsilon^2t^2|)I&[2c\varepsilon^3t^3+\nu \varepsilon^2 t^2+2(1-\varepsilon)\varepsilon t]I\\
[2c\varepsilon^3t^3+\nu \varepsilon^2 t^2+2(1-\varepsilon)\varepsilon t]I&[2c \varepsilon^2t^2+4\varepsilon \nu t-2\varepsilon]I
\end{pmatrix}\\ \geq \begin{pmatrix}
0&
0\\
0&[2c \varepsilon^2t^2+4\varepsilon \nu t-2\varepsilon]I-\frac{[2c\varepsilon^2t^2+\nu \varepsilon t+2(1-\varepsilon)]^2}{2(1-3\varepsilon-\varepsilon|1-2\varepsilon^2t^2|)}I
\end{pmatrix}.
\end{multline*}
Using this matrix inequality,  we obtain from \eqref{QP+PQ^T.V}:
\begin{multline}\label{QP+pQ^T.V1}
-\int_{\mathbb{R}^{2n}} u^T\left\{QP+PQ^T-\partial_t P\right\}uf_{\infty}dxdv\\
\leq \int_{\mathbb{R}^{2n}}  |u_2|^2\left[\frac{|1-2\varepsilon^3t^2|}{2\varepsilon^3}\left|\left|\frac{\partial^2V}{\partial x^2}+c I\right|\right|^2+|1-2\nu t+2\varepsilon^2 t^2|\left|\left|\frac{\partial^2V}{\partial x^2}+c I\right|\right|\right. \\ \left. -2c\varepsilon^2t^2-4\varepsilon \nu t+2\varepsilon+\frac{[2c\varepsilon^2t^2+\nu \varepsilon t+2(1-\varepsilon)]^2}{2(1-3\varepsilon-\varepsilon|1-2\varepsilon^2t^2|)}\right]
f_{\infty}dxdv.
\end{multline}
\eqref{derv.S.hypel.}, \eqref{second.deriv.}, \eqref{est.mix.V}, and \eqref{QP+pQ^T.V1} show that  
\begin{multline*}
\frac{d}{dt}\int_{\mathbb{R}^{2n}}u^TPuf_{\infty}dxdv 
\leq -2t\sigma\int_{\mathbb{R}^{2n}}\left\{\sum_{i=1}^n (\partial_{v_i}u_2)^T\left(\frac{\partial^2V}{\partial x^2}+c I\right)
\partial_{v_i}u_2\right\}f_{\infty}dxdv\\
+\frac{2t\sigma}{\nu}\int_{\mathbb{R}^{2n}}\left\{\sum_{i=1}^n (\partial_{v_i}u_{2})^T \frac{\partial^2(\partial_{x_i}V)}{\partial x^2} u_{2} \right\}f_{\infty}dxdv\\+\int_{\mathbb{R}^{2n}}  |u_2|^2\left[\frac{|1-2\varepsilon^3t^2|}{2\varepsilon^3}\left|\left|\frac{\partial^2V}{\partial x^2}+c I\right|\right|^2+|1-2\nu t+2\varepsilon^2 t^2|\left|\left|\frac{\partial^2V}{\partial x^2}+c I\right|\right| \right. \\ \left.  -2c\varepsilon^2t^2-4\varepsilon \nu t+2\varepsilon+\frac{[2c\varepsilon^2t^2+\nu \varepsilon t+2(1-\varepsilon)]^2}{2(1-3\varepsilon-\varepsilon|1-2\varepsilon^2t^2|)}\right]
f_{\infty}dxdv.
\end{multline*}   
As the matrix \eqref{Condition1} is positive semi-definite, we have 
\begin{multline*}
-2t\sigma\int_{\mathbb{R}^{2n}}\left\{\sum_{i=1}^n (\partial_{v_i}u_2)^T\left(\frac{\partial^2V}{\partial x^2}+cI\right)
\partial_{v_i}u_2\right\}f_{\infty}dxdv
+\frac{2t\sigma}{\nu}\int_{\mathbb{R}^{2n}}\left\{\sum_{i=1}^n (\partial_{v_i}u_{2})^T \frac{\partial^2(\partial_{x_i}V)}{\partial x^2} u_{2} \right\}f_{\infty}dxdv\\ \leq \tau t\int_{\mathbb{R}^{2n}}u_2^T\left(\frac{\partial^2V}{\partial x^2}+cI\right)
u_2f_{\infty}dxdv \leq \tau t\int_{\mathbb{R}^{2n}}|u_2|^2\left|\left|\frac{\partial^2V}{\partial x^2}+c I\right|\right|
f_{\infty}dxdv.
\end{multline*}
 Subsequently, 
 \begin{multline}\label{der.S.V.}
\frac{d}{dt}\int_{\mathbb{R}^{2n}}u^TPuf_{\infty}dxdv 
\\ \leq \int_{\mathbb{R}^{2n}}  |u_2|^2\left[\frac{|1-2\varepsilon^3t^2|}{2\varepsilon^3}\left|\left|\frac{\partial^2V}{\partial x^2}+c I\right|\right|^2+(|1-2\nu t+2\varepsilon^2 t^2|+\tau t)\left|\left|\frac{\partial^2V}{\partial x^2}+c I\right|\right| \right. \\ \left.  -2c\varepsilon^2t^2-4\varepsilon \nu t+2\varepsilon+\frac{[2c\varepsilon^2t^2+\nu \varepsilon t+2(1-\varepsilon)]^2}{2(1-3\varepsilon-\varepsilon|1-2\varepsilon^2t^2|)}  \right]
f_{\infty}dxdv.
\end{multline} 
\subsubsection*{Step 4, decay of the functional $\Psi$:}   We estimate the time derivative of \eqref{func.t}:
Combining  \eqref{t.der.}, \eqref{t.d.f^2V1}, and \eqref{der.S.V.} yield
\begin{multline}
\frac{d}{dt}\Psi(t,f(t) )
\\ \leq - \int_{\mathbb{R}^{2n}}  |u_2|^2\left[\left(\sigma e^{-8\tau n t} \gamma_1-\frac{|1-2\varepsilon^3t^2|}{2\varepsilon^3} \right)\left|\left|\frac{\partial^2V}{\partial x^2}+c I\right|\right|^2-(|1-2\nu t+2\varepsilon^2 t^2|+\tau t)\left|\left|\frac{\partial^2V}{\partial x^2}+c I\right|\right| \right. \\ \left. +2\sigma \gamma_2 +2c\varepsilon^2t^2+4\varepsilon \nu t-2\varepsilon-\frac{[2c\varepsilon^2t^2+\nu \varepsilon t+2(1-\varepsilon)]^2}{2(1-3\varepsilon-\varepsilon|1-2\varepsilon^2t^2|)}\right]
f_{\infty}dxdv.
\end{multline}
  We fix $\gamma_1>0$ and $\gamma_2>0$ such that \begin{multline}\label{aux.eq.5}
 \left(\sigma e^{-8\tau n t } \gamma_1-\frac{|1-2\varepsilon^3t^2|}{2\varepsilon^3} \right)\left|\left|\frac{\partial^2V}{\partial x^2}+c I\right|\right|^2-(|-1+2\nu t-2\varepsilon^2 t^2|+\tau t)\left|\left|\frac{\partial^2V}{\partial x^2}+c I\right|\right|\\ +2\sigma \gamma_2 +2c\varepsilon^2t^2+4\varepsilon \nu t-2\varepsilon-\frac{[2c\varepsilon^2t^2+\nu \varepsilon t+2(1-\varepsilon)]^2}{2(1-3\varepsilon-\varepsilon|1-2\varepsilon^2t^2|)}\geq 0
 \end{multline} 
for all $x \in \mathbb{R}^n$ and $t\in [0,t_0].$   We recall that we have fixed $\varepsilon=\varepsilon(t_0)$  so that \eqref{eps(t_0)} holds, which makes the above denominator positive. The existence of such $\gamma_1>0$ and $\gamma_2>0$  can be proven by the following arguments: We can consider the left hand side of \eqref{aux.eq.5} as a quadratic polynomial of $\left|\left|\frac{\partial^2V}{\partial x^2}+c I\right|\right|\in [0, \infty).$ As time $t$ varies in a bounded interval $[0,t_0],$ the terms containing $t$ are bounded. Therefore, we can choose large values for $\gamma_1=\gamma_1(t_0)$ and $\gamma_2=\gamma_2(t_0)$ so that  this quadratic polynomial is non-negative  for all $t\in [0,t_0].$

Consequently, we obtain that  
$$
\frac{d}{dt}\Psi(t, f(t))\leq 0. $$
Hence $
\Psi(t, f(t)) $ is decreasing and  
\begin{equation}\label{psi0}
\Psi(t,f(t))\leq \Psi(0,f_0) \, \, \, \, \, \text{ for all } \,\, \, \, t\in [0,t_0].
\end{equation} 
\eqref{func.t.pos.} and \eqref{psi0} show that  
\begin{equation}\label{u-1}
\displaystyle \int_{\mathbb{R}^{2n}}|u_1|^2 f_{\infty} dx dv \leq \frac{1}{ \varepsilon ^3t^3}\int_{\mathbb{R}^{2n}} \left(\frac{f_0}{f_{\infty}} -1\right)^2 \left(\gamma_1  \left|\left|\frac{\partial^2V}{\partial x^2}+c I\right|\right|^2+\gamma_2\right) f_{\infty}dxdv,
\end{equation}
\begin{equation}\label{u-2}
\int_{\mathbb{R}^{2n}} |u_2|^2 f_{\infty} dxdv\leq  \frac{1}{\varepsilon t}\int_{\mathbb{R}^{2n}} \left(\frac{f_0}{f_{\infty}} -1\right)^2 \left(\gamma_1  \left|\left|\frac{\partial^2V}{\partial x^2}+c I\right|\right|^2+\gamma_2\right) f_{\infty}dxdv,
\end{equation} and 
\begin{equation}\label{u-2 V}
 \int_{\mathbb{R}^{2n}} u_2^T \left(\frac{\partial^2V}{\partial x^2}+cI\right) u_2 f_{\infty} dxdv \leq \frac{1}{t}\int_{\mathbb{R}^{2n}} \left(\frac{f_0}{f_{\infty}}-1 \right)^2 \left(\gamma_1  \left|\left|\frac{\partial^2V}{\partial x^2}+c I\right|\right|^2+\gamma_2\right) f_{\infty}dxdv.
\end{equation}
It is clear that there is a positive constant $C$ such that \begin{equation}\label{V..}\gamma_1 \left|\left|\frac{\partial^2V}{\partial x^2}+c I\right|\right|^2+\gamma_2\leq C\left(  \left| \left|\frac{\partial^2V}{\partial x^2}\right|\right|^2+1\right).
\end{equation}
\eqref{u-1},  a proper linear combination of \eqref{u-2} and \eqref{u-2 V}, and \eqref{V..} imply the claimed estimates \eqref{Vhyp.ell.1}, \eqref{Vhyp.ell.2}.
 \end{proof}

\bigskip

\begin{proof}[{{Proof of Corollary \ref{Cor.}}}]
Theorem \ref{Main} and Theorem \ref{Hyp.ellip.} show that, for $t\geq t_0>0,$ 
\begin{multline}\label{t_01} 
\int_{\mathbb{R}^{2n}}\left(\frac{f(t)}{f_{\infty}}-1\right)^2 f_{\infty}dxdv +\int_{\mathbb{R}^{2n}}\left|\nabla_{x}\left(\frac{f(t)}{f_{\infty}}\right)\right|^2 f_{\infty}dxdv\\+\int_{\mathbb{R}^{2n}}\nabla^T_{v}\left(\frac{f(t)}{f_{\infty}}\right)\left(\frac{\partial^2 V}{\partial x^2}+(1-\alpha_0)I\right)\nabla_{v}\left(\frac{f(t)}{f_{\infty}}\right)f_{\infty} dxdv \\
   \leq C e^{-2\lambda (t-t_0)} \left[\int_{\mathbb{R}^{2n}}\left(\frac{f(t_0)}{f_{\infty}}-1\right)^2 f_{\infty}dxdv+\int_{\mathbb{R}^{2n}}\left|\nabla_{x}\left(\frac{f(t_0)}{f_{\infty}}\right)\right|^2 f_{\infty}dxdv\right.\\+ \left.\int_{\mathbb{R}^{2n}}\nabla^T_{v}\left(\frac{f(t_0)}{f_{\infty}}\right)\left(\frac{\partial^2 V}{\partial x^2}+ (1-\alpha_0) I\right)\nabla_{v}\left(\frac{f(t_0)}{f_{\infty}}\right)f_{\infty}dxdv\right]
  \end{multline} 
 holds with the constant $C$ and the rate $\lambda$ given in Theorem \ref{Main}.
Using \eqref{Vhyp.ell.1} and \eqref{Vhyp.ell.2} at $t=t_0,$ we get 
\begin{equation}\label{t_02}
 \displaystyle \int_{\mathbb{R}^{2n}}\left|\nabla_x \left(\frac{f(t_0)}{f_{\infty}}\right) \right|^2f_{\infty} dx dv \leq \frac{C_1}{t_0^{3}} \int_{\mathbb{R}^{2n}} \left(\frac{f_0}{f_{\infty}} -1\right)^2 \left(\left|\left|\frac{\partial^2V}{\partial x^2}\right|\right|^2+1\right) f_{\infty}dxdv
\end{equation}
and 
\begin{multline}\label{t_03}
 \displaystyle \int_{\mathbb{R}^{2n}}\nabla_v^T \left(\frac{f(t_0)}{f_{\infty}}\right)\left(\frac{\partial^2V}{\partial x^2}+(1-\alpha_0)I\right)\nabla_v \left(\frac{f(t_0)}{f_{\infty}}\right)  f_{\infty} dx dv\\ \leq \frac{C_2}{t_0}\int_{\mathbb{R}^{2n}} \left(\frac{f_0}{f_{\infty}} -1\right)^2 \left(\left|\left|\frac{\partial^2V}{\partial x^2}\right| \right|^2+1\right) f_{\infty}dxdv.
\end{multline}
Combining \eqref{t_01}, \eqref{t_02}, and \eqref{t_03}, we obtain 
\eqref{main.est.hyp}. 
\end{proof}
\bigskip

\section{Appendix}
\subsection{ Proof that Assumption \textcolor{blue}{2.}\ref{A2}\textcolor{blue}{'} implies Assumption \ref{A1}}\label{6.1}
 Assume Assumption \textcolor{blue}{2.}\ref{A2}' is satisfied. Let  $(u_1, u_2,...,u_{n+1})^T$  be any vector in $  \mathbb{R}^{n(n+1)},$ where $u_i $ is a vector in $ \mathbb{R}^n$ for all $i\in\{1,...,n+1\}. $ 
We compute the quadratic form of the matrix \eqref{Condition1}
$$\begin{pmatrix}
u_1\\ u_2\\.\\.\\.\\ u_{n+1}
\end{pmatrix}^T \begin{pmatrix}
 \nu \left(\frac{\partial^2  V(x)}{\partial x^2}+cI\right)&0&...&0&-\frac{1}{2}\frac{\partial^2(\partial_{x_1}V(x))}{\partial x^2} \\
 0&\nu \left(\frac{\partial^2  V(x)}{\partial x^2}+cI\right)&...&0&-\frac{1}{2}\frac{\partial^2(\partial_{x_2}V(x))}{\partial x^2}\\ 
 ...&...&...&...&...\\
 0&0&...&\nu \left(\frac{\partial^2  V(x)}{\partial x^2}+cI\right)&-\frac{1}{2}\frac{\partial^2(\partial_{x_n}V(x))}{\partial x^2}\\
 -\frac{1}{2}\frac{\partial^2(\partial_{x_1}V(x))}{\partial x^2}&-\frac{1}{2}\frac{\partial^2(\partial_{x_2}V(x))}{\partial x^2}&...&-\frac{1}{2}\frac{\partial^2(\partial_{x_n}V(x))}{\partial x^2}&\frac{\tau \nu}{2\sigma}\left(\frac{\partial^2  V(x)}{\partial x^2}+cI\right)
 \end{pmatrix}\begin{pmatrix}
u_1\\ u_2\\.\\.\\.\\ u_{n+1}
\end{pmatrix}$$
$$=\sum_{i=1}^n\left\{ \nu u_i^T\left(\frac{\partial^2  V(x)}{\partial x^2}+cI\right)u_i-u^T_{i}\frac{\partial^2  (\partial_{x_i}V(x))}{\partial x^2}u_{n+1}\right\}+\frac{\tau \nu}{2\sigma}u_{n+1}^T\left(\frac{\partial^2  V(x)}{\partial x^2}+cI\right)u_{n+1}.$$ To show that \eqref{Condition1} is positive semi-definite, it is enough to show the quadratic form above is non-negative.  Assumption \textcolor{blue}{2.}\ref{A2}' implies
 \begin{equation*}
 \left| u^T_{i}\frac{\partial^2  (\partial_{x_i}V(x))}{\partial x^2}u_{n+1}\right|\leq |u_{i}||u_{n+1}|\sqrt{\frac{2\tau \nu^2}{n\sigma}}(\alpha(x)+ c)\leq \nu (\alpha(x)+ c) |u_i|^2+{\frac{ \tau \nu}{2n\sigma}}(\alpha(x)+ c)|u_{n+1}|^2. 
 \end{equation*} Therefore, we get the desired result
$$\sum_{i=1}^n\left\{ \nu u_i^T\left(\frac{\partial^2  V(x)}{\partial x^2}+cI\right)u_i-u^T_{i}\frac{\partial^2  (\partial_{x_i}V(x))}{\partial x^2}u_{n+1}+\frac{\tau \nu}{2n\sigma}u_{n+1}^T\left(\frac{\partial^2  V(x)}{\partial x^2}+cI\right)u_{n+1}\right\} $$
$$\geq \sum_{i=1}^n\left\{ \nu u_i^T\left(\frac{\partial^2  V(x)}{\partial x^2}-\alpha(x)I\right)u_i+\frac{\tau \nu}{2n\sigma}u_{n+1}^T\left(\frac{\partial^2  V(x)}{\partial x^2}-\alpha(x) I\right)u_{n+1}\right\}\geq 0. $$ 
$$ \, \, \, \, \, \, \, \, \, \, \, \, \, \, \, \, \, \,\, \, \, \, \, \, \, \, \, \, \,  \, \, \, \, \, \, \, \, \, \, \, \, \, \, \, \, \, \, \, \, \, \, \, \, \, \, \, \, \, \, \, \, \, \, \, \, \, \, \, \, \, \, \, \, \, \, \, \, \, \, \, \, \, \, \, \, \, \, \, \, \, \, \, \, \, \, \, \, \, \, \, \, \, \, \, \, \, \, \, \, \, \, \, \, \, \, \, \, \, \, \, \, \, \, \, \, \, \, \, \, \, \, \, \, \, \, \, \, \, \, \, \, \, \, \, \, \, \, \, \, \, \, \, \, \, \, \, \, \, \, \, \, \, \, \, \, \, \, \, \, \, \, \, \, \, \, \, \, \, \, \, \, \, \, \, \, \, \, \, \, \, \, \, \, \, \, \, \, \, \, \, \, \, \, \, \, \, \, \, \, \, \, \, \, \, \, \, \, \, \, \, \, \, \, \, \, \, \, \, \, \, \, \, \, \, \, \, \, \, \, \, \, \, \, \, \, \, \, \, \, \, \, \, \, \, \square $$

\subsection{ Matrix inequalities for Section 5.1}\label{6.2}

\begin{lemma}\label{lem:mat.inq.}
Let $\alpha_0>-\infty$ be the constant defined by \eqref{alpha-zero},  $a\in \mathbb{R}$ be some constant such that $a+\alpha_0>\frac{\nu^2}{4},$ and $P\colonequals \begin{pmatrix}
 2I&\nu I\\\nu I& 2\frac{\partial^2 V}{\partial x^2}+2a I
\end{pmatrix}.$ Then \begin{equation}\label{c_1P<()<c_2P}
c_1 P\leq \begin{pmatrix}
 I&0\\0& \frac{\partial^2 V}{\partial x^2}+(1-\alpha_0)I
\end{pmatrix}\leq c_2 P
\end{equation} holds with  $c_1 \colonequals \frac{1}{a+\alpha_0+1+\sqrt{(a+\alpha_0-1)^2+\nu^2}}>0, \, \, \, \, \, c_2 \colonequals \frac{a+\alpha_0+1+\sqrt{(a+\alpha_0-1)^2+\nu^2}}{4(a+\alpha_0)-\nu^2}>0.$
\end{lemma}
\begin{proof}
We consider, for some $k\in \mathbb{R}$ to be chosen later as $\frac{1}{2c_{1,2}},$
\begin{equation*}
A\colonequals P-2k\begin{pmatrix}
 I&0\\0& \frac{\partial^2 V}{\partial x^2}+(1-\alpha_0)I
\end{pmatrix}=\begin{pmatrix}
 2(1-k)I&\nu I\\\nu I& 2(1-k)\left(\frac{\partial^2 V}{\partial x^2}+(1-\alpha_0)I\right)+2(a+\alpha_0-1)I
\end{pmatrix}.
\end{equation*}
We check the (real) eigenvalues $\eta$ of the symmetric matrix $A$ (depending on $k$). It is easy to check that $\eta= 2(1-k)$ is not an eigenvalue of $A.$ If $\eta\neq 2(1-k),$ then we have the condition
\begin{multline*}
\det(A-\eta I)=\begin{vmatrix}
2(1-k)I-\eta I&\nu I\\\nu I& 2(1-k)\left(\frac{\partial^2 V}{\partial x^2}+(1-\alpha_0)I\right)+2(a+\alpha_0-1)I-\eta I
\end{vmatrix}\\
=\begin{small}\frac{1}{(2(1-k)-\eta)^n}\begin{vmatrix}
2(1-k)I-\eta I&0\\ \nu I& (2(1-k)-\eta)\left[2(1-k)\left(\frac{\partial^2 V}{\partial x^2}+(1-\alpha_0)I\right)+2(a+\alpha_0-1)I-\eta I\right]-\nu^2 I
\end{vmatrix}\end{small}\\
=\det\left((2(1-k)-\eta)\left[2(1-k)\left(\frac{\partial^2 V}{\partial x^2}+(1-\alpha_0)I\right)+2(a+\alpha_0-1)I-\eta I\right]-\nu^2 I\right) =0.
\end{multline*}
If $\alpha_i,$ $i\in\{1,...,n\}$ are the eigenvalues of $\frac{\partial^2 V}{\partial x^2},$    then the eigenvalues $\eta$ of $A$ satisfy
\begin{equation}\label{product}
\prod_{i=1}^{n} \left(\eta^2-2\eta[(1-k)(\alpha_i-\alpha_0+2)+a+\alpha_0-1]+4(1-k)^2(\alpha_i-\alpha_0+1)+4(1-k)(a+\alpha_0-1)-\nu^2 \right)=0.
\end{equation}

\textbf{Right inequality of \eqref{c_1P<()<c_2P}:}
From \eqref{product}, we see that
$A$ is positive semi-definite (i.e., all $\eta\geq 0$) if the following three conditions hold:
\begin{equation}\label{k1}
1-k\geq 0, \, \, \, \, \, \, \text{ (due to the first minor of $A$)}
\end{equation}
\begin{equation}\label{k2}
(1-k)(\alpha_i-\alpha_0+2)+a+\alpha_0-1\geq 0, \, \, \, \forall i\in\{1,...,n\},
\end{equation}
\begin{equation}\label{k3}
4(1-k)^2(\alpha_i-\alpha_0+1)+4(1-k)(a+\alpha_0-1)-\nu^2 \geq 0, \, \, \, \, \forall i\in\{1,...,n\}.
\end{equation}
We set  $$k\colonequals \frac{1}{2c_2}>0.$$
Then, \eqref{k1} holds: 
\begin{equation}\label{1-k}
1-k=\frac{\sqrt{(a+\alpha_0-1)^2+\nu^2}-(a+\alpha_0-1)}{2}>0.
\end{equation}
 Using $\alpha_i\geq \alpha_0$ for all $i\in \{1,...,n\}$ we see that \eqref{k2} also holds: $$(1-k)(\alpha_i-\alpha_0+2)+a+\alpha_0-1\geq 2(1-k)+a+\alpha_0-1=\sqrt{(a+\alpha_0-1)^2+\nu^2}>0. $$
To verify \eqref{k3} we estimate using $\alpha_i\geq \alpha_0$ for all $i\in \{1,...,n\}$ and \eqref{1-k} 
 \begin{multline*}\label{quadp}
4(1-k)^2(\alpha_i-\alpha_0+1)+4(1-k)(a+\alpha_0-1)-\nu^2\\\geq 4(1-k)^2+4(1-k)(a+\alpha_0-1)-\nu^2=0.\end{multline*}
Therefore, for $k$ defined in \eqref{1-k}, $A$ is positive semi-definite. Hence, the inequality on the right hand side of \eqref{c_1P<()<c_2P} holds.

\textbf{Left inequality of \eqref{c_1P<()<c_2P}:} Similarly, $A$ is negative semi-definite if the following three conditions hold:\begin{equation}\label{k1-}
1-k\leq 0,
\end{equation}
\begin{equation}\label{k2-}
(1-k)(\alpha_i-\alpha_0+2)+a+\alpha_0-1\leq 0, \, \, \, \forall i\in\{1,...,n\},
\end{equation}
\begin{equation}\label{k3-}
4(1-k)^2(\alpha_i-\alpha_0+1)+4(1-k)(a+\alpha_0-1)-\nu^2 \geq 0, \, \, \, \, \forall i\in\{1,...,n\}.
\end{equation}
Setting $$
k\colonequals \frac{1}{2c_1}>0 $$
we find 
\begin{equation}\label{k-}
1-k=\frac{-\sqrt{(a+\alpha_0-1)^2+\nu^2}-(a+\alpha_0-1)}{2}<0
\end{equation}
and 
$$(1-k)(\alpha_i-\alpha_0+2)+a+\alpha_0-1\leq 2(1-k)+a+\alpha_0-1=-\sqrt{(a+\alpha_0-1)^2+\nu^2}<0. $$
Finally, we check using $\alpha_i\geq \alpha_0$ for all $i\in \{1,...,n\}$ and \eqref{k-}  $$4(1-k)^2(\alpha_i-\alpha_0+1)+4(1-k)(a+\alpha_0-1)-\nu^2
\geq 4(1-k)^2+4(1-k)(a+\alpha_0-1)-\nu^2=0.$$
 Therefore, for $k$ defined in \eqref{k-}, $A$ is negative semi-definite. Hence, the inequality on the left hand side of \eqref{c_1P<()<c_2P} holds.
\end{proof}

\begin{Remark}
Lemma \ref{lem:mat.inq.} proves the following matrix inequalities from Section 5.1:
\begin{itemize}
\item[(a)]If $a=0$ and $\alpha_0>\frac{\nu^2}{4},$ then \eqref{c_1P<()<c_2P} is the matrix inequality \eqref{mat.in1}.
\item[(b)] If $a=\frac{\varepsilon^2}{2}$ and $\alpha_0=\frac{\nu^2}{4},$ then \eqref{c_1P<()<c_2P} is the matrix inequality \eqref{mat.in2}.
\item[(c)] \eqref{c_1P<()<c_2P} coincides with the matrix inequality \eqref{mat.in3}.
\end{itemize}
\end{Remark}
\subsection{ Proof of  inequality \eqref{aA}}\label{6.3}
We recall the assumption $\alpha_0< \frac{\nu^2}{4}.$ We first rewrite  
$$\displaystyle A_2^{-1}=\frac{2\nu \min\{1, \alpha_0\}}{1+ \frac{\nu^2}{2}-\alpha_0+ \sqrt{( \frac{\nu^2}{2}-\alpha_0-1)^2+\nu^2}}=\frac{4 \min\{1, \alpha_0\}}{{\nu}+2(1-\alpha_0)\nu^{-1}+ \sqrt{(\nu^2-4\alpha_0)+4(\alpha_0+1)^2\nu^{-2}}}, $$
$$\nu -\sqrt{\nu^2-4 \alpha_0}=\frac{4\alpha_0}{\nu+\sqrt{\nu^2-4\alpha_0}}.$$
 Then \eqref{aA} is equivalent to 
 \begin{equation}\label{aaAA}
 \frac{\alpha_0}{\nu+\sqrt{\nu^2-4\alpha_0}}\geq \frac{\min\{1, \alpha_0\}}{\nu+2(1-\alpha_0)\nu^{-1}+ \sqrt{(\nu^2-4\alpha_0)+4(\alpha_0+1)^2\nu^{-2}}}.\end{equation}
 
If $\min\{1, \alpha_0\}=\alpha_0,$ then \eqref{aaAA} is true because of $$\nu+2(1-\alpha_0)\nu^{-1}+ \sqrt{(\nu^2-4\alpha_0)+4(\alpha_0+1)^2\nu^{-2}}>\nu+\sqrt{\nu^2-4\alpha_0}. $$

If $\min\{1, \alpha_0\}=1,$ then \eqref{aaAA} is equivalent to $$\alpha_0 \nu-2\alpha_0(\alpha_0-1)\nu^{-1}+ \alpha_0\sqrt{(\nu^2-4\alpha_0)+4(\alpha_0+1)^2\nu^{-2}}\geq \nu+\sqrt{\nu^2-4\alpha_0}, $$
or equivalently
$$(\alpha_0-1) (\nu^2-2\alpha_0)\nu^{-1}+ \alpha_0\sqrt{(\nu^2-4\alpha_0)+4(\alpha_0+1)^2\nu^{-2}}\geq \sqrt{\nu^2-4\alpha_0}. $$
The last inequality holds since $$(\alpha_0-1) (\nu^2-2\alpha_0)\nu^{-1}\geq (\alpha_0-1) (\nu^2-4\alpha_0)\nu^{-1}\geq 0$$ and $$\alpha_0\sqrt{(\nu^2-4\alpha_0)
+4(\alpha_0+1)^2\nu^{-2}}>\sqrt{\nu^2-4\alpha_0}.$$ 

These two cases show that inequality \eqref{aA} holds. 
$$ \, \, \, \, \, \, \, \, \, \, \, \, \, \, \, \, \, \,\, \, \, \, \, \, \, \, \, \, \,  \, \, \, \, \, \, \, \, \, \, \, \, \, \, \, \, \, \, \, \, \, \, \, \, \, \, \, \, \, \, \, \, \, \, \, \, \, \, \, \, \, \, \, \, \, \, \, \, \, \, \, \, \, \, \, \, \, \, \, \, \, \, \, \, \, \, \, \, \, \, \, \, \, \, \, \, \, \, \, \, \, \, \, \, \, \, \, \, \, \, \, \, \, \, \, \, \, \, \, \, \, \, \, \, \, \, \, \, \, \, \, \, \, \, \, \, \, \, \, \, \, \, \, \, \, \, \, \, \, \, \, \, \, \, \, \, \, \, \, \, \, \, \, \, \, \, \, \, \, \, \, \, \, \, \, \, \, \, \, \, \, \, \, \, \, \, \, \, \, \, \, \, \, \, \, \, \, \, \, \, \, \, \, \, \, \, \, \, \, \, \, \, \, \, \, \, \, \, \, \, \, \, \, \, \, \, \, \, \, \, \, \, \, \, \, \, \, \, \, \, \, \, \, \, \, \square $$


\bigskip
\bigskip
\textbf{Acknowledgement.} Both authors acknowledge support by the Austrian Science Fund (FWF) project \href{https://doi.org/10.55776/F65}{10.55776/F65}.
We are also grateful to the anonymous referees, whose suggestions helped to improve this paper.
\bigskip

\textbf{Data Availability.} Data will be made available on reasonable request.
\bigskip

\textbf{Conflict of interest.} The authors have no conflict of interest to declare.


\end{document}